\documentclass[11pt]{amsart}
\usepackage[utf8]{inputenc}
\usepackage[margin=1.25in]{geometry}
\usepackage{amsmath,amssymb,amsthm}
\usepackage{tikz}
\usepackage{tikz-cd}    
\usepackage{rotating}
\usepackage{commath}    
\usepackage{mathtools}    
\usepackage{graphicx}
\usepackage{etoolbox}
\usepackage{enumitem}
\usepackage{comment}
\usepackage{stmaryrd}
\usepackage{hyperref}
\usepackage{mathrsfs}

\usepackage[
backend=biber,
style=alphabetic,
maxalphanames=10,
maxbibnames=10,
sorting=anyt,
]{biblatex}

\theoremstyle{plain}
\newtheorem{thm}{Theorem}[subsection]
\newtheorem{lem}[thm]{Lemma}
\newtheorem{prop}[thm]{Proposition}
\newtheorem{cor}[thm]{Corollary}

\theoremstyle{definition}
\newtheorem{defn}[thm]{Definition}
\newtheorem{eg}[thm]{Example}

\theoremstyle{remark}
\newtheorem{rem}[thm]{Remark}

\numberwithin{equation}{section}

\newcommand{\ZZ}{\mathbb{Z}}
\newcommand{\NN}{\mathbb{N}}
\newcommand{\QQ}{\mathbb{Q}}

\newcommand{\CC}{\mathbb{C}}

\newcommand{\FF}{\mathbb{F}}
\newcommand{\GG}{\mathbb{G}}

\newcommand{\mfk}{\mathfrak{m}}
\newcommand{\nfk}{\mathfrak{n}}

\newcommand{\Pfk}{\mathfrak{P}}

\newcommand{\Gcal}{\mathcal{G}}

\newcommand{\Mcal}{\mathcal{M}}

\newcommand{\Ocal}{\mathcal{O}}

\newcommand{\Zcal}{\mathcal{Z}}

\newcommand{\End}{\operatorname{End}}

\newcommand{\Gal}{\operatorname{Gal}}

\newcommand{\Tr}{\operatorname{Tr}}

\newcommand{\Nr}{\operatorname{Nr}}
\newcommand{\sgn}{\operatorname{sgn}}
\newcommand{\Res}{\operatorname{Res}}
\newcommand{\ord}{\operatorname{ord}}

\newcommand{\Frob}{\operatorname{Frob}}

\newcommand{\ovl}{\overline}

\newcommand{\td}{\tilde}

\newcommand{\sbe}{\subseteq}

\newcommand{\inj}{\hookrightarrow}

\let\oldforall\forall
\renewcommand{\forall}{\oldforall \: }
\let\oldexist\exists
\renewcommand{\exists}{\oldexist \: }

\newcommand{\floor}[1]{\lfloor #1 \rfloor}

\newcommand{\ang}[1]{\langle #1 \rangle}
\newcommand{\anginf}[1]{\langle #1 \rangle_\infty}

\newcommand{\lrang}[1]{\left\langle #1 \right\rangle}
\newcommand{\lranginf}[1]{\left\langle #1 \right\rangle_\infty}

\newcommand{\Mod}[1]{\ (\mathrm{mod}\ #1)}

\newcommand{\Fq}{\FF_q}
\newcommand{\Fqst}{\FF_q^\times}
\newcommand{\Fqd}{\FF_{q^d}}

\newcommand{\Fqdl}{\FF_{q^{d\l}}}
\newcommand{\T}{\theta}
\newcommand{\Ami}{A_{+,i}}

\newcommand{\vaf}{\Pi_v^{\textnormal{ari}}}
\newcommand{\vag}{\Gamma_v^{\textnormal{ari}}}

\newcommand{\vgf}{\Pi_v^{\textnormal{geo}}}
\newcommand{\vgg}{\Gamma_v^{\textnormal{geo}}}

\newcommand{\vtf}{\Pi_v}
\newcommand{\vtg}{\Gamma_v}

\newcommand{\bggs}{\Gcal^{\textnormal{geo}}}
\newcommand{\bags}{\Gcal^{\textnormal{ari}}}
\newcommand{\ggs}{G_\l^{\textnormal{geo}}}
\newcommand{\ggsf}{G_f^{\textnormal{geo}}}
\newcommand{\ags}{G_\l^{\textnormal{ari}}}

\let\l\ell

\DeclarePairedDelimiter{\angres}{\langle}{\rangle_{\textnormal{Res}}}
\DeclarePairedDelimiter{\angtr}{\langle}{\rangle_\textnormal{Tr}}

\makeatletter

\newcommand{\myToC}{{
		\renewcommand{\contentsname}{Contents}
		\@starttoc{toc}{\contentsname}
}}

\patchcmd{\@tocline}
{\hfil}
{\leaders\hbox{\,.\,}\hfil}

\makeatother

\hypersetup{
	colorlinks=true,
	linkcolor=blue,
	citecolor=red,
	linktoc=page,
	linkbordercolor=blue,
	citebordercolor=red,
}

\title[Geometric Gauss Sums and Gross-Koblitz Formula over Function Fields]{Geometric Gauss Sums and Gross-Koblitz Formula \\ over Function Fields}
\author{Ting-Wei Chang}
\date{\today}

\thanks{This work is supported by the National Science and Technology Council grant no. 109-2115-M-007-017-MY5 and 113-2628-M-007-003.}

\begin{document}
	
	\begin{abstract}
		In this paper, we introduce an analog of Gauss sums over function fields in positive characteristic.
		We establish several fundamental properties, including reflection formula, Stickelberger's theorem, and Hasse-Davenport relations.
		In addition, we determine their absolute values and signs at infinity.
		While these results parallel the classical theory of Gauss sums as well as Thakur’s “arithmetic” analogs over function fields, our approach differs completely from both of the preceding cases.
		Specifically, we first prove a Gross-Koblitz-type formula relating geometric Gauss sums to special $v$-adic gamma values.
		The properties of geometric Gauss sums then follow from the specializations of this formula together with the functional equations of $v$-adic gamma functions.
	\end{abstract}
	
	\maketitle
	
	\setcounter{tocdepth}{2}
	\myToC
	
	\section{Introduction}    \label{section-introduction}
	
	\subsection{Classical Gauss sums}     \label{subsection-classical-gauss-sums}
	
	In this paper, we introduce an analog of Gauss sums in positive characteristic.
	Classically, Gauss sums are finite sums of roots of unity arising from a multiplicative and an additive character on finite fields.
	Thus, they are regarded as “finite” analogs of Euler's gamma function, where two characters correspond respectively to $t^s$ and $e^t$.
	Gauss sums are fundamental objects in number theory and arise in different areas.
	For example, they appear in the annihilation of the ideal class group of cyclotomic fields by Stickelberger's theorem \cite{stickelberger1890ueber} (see also \cite[Chapter 1]{lang1990cyclotomic}, \cite[Chapter 6]{washington1997introduction}), as special $p$-adic gamma values by Gross-Koblitz formula \cite{gk1979gauss}, and as Jacobi sum Hecke characters by the work of Weil \cite{weil1952jacobi}.
	In what follows, we outline the aspects of Gauss sums that will be central to this paper.
	
	As complex numbers, Gauss sums lie uniformly on circles in the complex plane, with radii completely determined by the cardinality of the given finite field.
	Thus, to determine their exact values, it is enough to evaluate their “signs”.
	Nevertheless, this question is only answered for particular cases of Gauss sums and remains open in general.
	We refer readers to \cite{be1981determination}.
	On the other hand, Gauss sums exhibit rich symmetry and compatibility in view of the reflection and the Hasse-Davenport product and lifting formulas \cite{hd1935dienullstellen}.
	These properties further strengthen the analogy between Gauss sums and Euler’s gamma function.
	
	As algebraic numbers, the prime factorizations of Gauss sums as principal ideals are described by Stickelberger's theorem, via particular integral group ring elements known as Stickelberger elements.
	This, in turn, implies that the Stickelberger elements annihilate the ideal class groups of cyclotomic fields.
	This phenomenon was later studied for arbitrary abelian extensions of global fields.
	In this context, the Brumer-Stark conjecture predicts that the so-called Stickelberger-Brumer elements, defined as special values of the $L$-function evaluators, annihilate the corresponding ideal class group.
	For further details, see \cite[Chapter 15]{rosen2002number}.
	
	In the $p$-adic setting, where $p$ is an odd prime number, Gauss sums can be expressed as product of specific Morita's $p$-adic gamma values \cite{morita1975padic}.
	This is due to the technical calculation of Gross and Koblitz \cite{gk1979gauss} (based on an unpublished work of Katz expressing Gauss sums as limits of particular factorials; see \cite[Chapter 15]{lang1990cyclotomic} for historical review), which is now called Gross-Koblitz formula.
	As a result, we obtained a class of algebraicity and monomial relations among $p$-adic gamma values in terms of Gauss sums.
	
	To this point, we have seen the importance of Gauss sums from various perspectives.
	The aim of this paper is to investigate similar phenomena in the setting of function fields in positive characteristic.
	Specifically, we introduce the notion of \textit{geometric Gauss sums} and develop their fundamental properties in parallel with the classical theory.
	
	\subsection{Gauss sums in positive characteristic}
	
	Let $p$ be any prime number and $q$ be a power of $p$.
	Let $A:=\Fq[\T]$ be the polynomial ring in the variable $\T$ over a finite field $\Fq$ and $k := \Fq(\T)$ be its field of fractions with a fixed algebraic closure $\ovl{k}$.
	We let $A_+$ be the set of all monic polynomials in $A$ and $\Ami$ be its subset consisting of all monic polynomials of degree $i \geq 0$.
	We fix an irreducible polynomial $v \in A_{+,d}$ where $d>0$.
	
	\subsubsection{Thakur's arithmetic Gauss sums}     \label{subsubsection-thakurs-arithmetic-gauss-sums}
	
	In his seminal paper \cite{thakur1988gauss}, Thakur developed a theory of Gauss sums over $A$.
	In his setting, Gauss sums are defined as finite “character” sums involving usual roots of unity in $\Fqd \sbe \ovl{k}$ and $v$-torsions of the Carlitz module, the latter serving as function field analogs of roots of unity.
	We will recall the precise definition in \eqref{eq-ari-gauss-sum}.
	Thakur's Gauss sums not only fit into the frameworks of all the classical theories outlined in \S\ref{subsection-classical-gauss-sums}, but also inspired a wide range of subsequent developments and generalizations, including \cite{thakur1991gauss}, \cite{thakur1993shtukas}, \cite{ap2015universal}, \cite{gm2021special}, and \cite{cwy2024vadic}.
	
	In particular, Thakur \cite[Theorem VI]{thakur1988gauss} (see also \eqref{eq-gkt-formula-for-ari}) established an analog of Gross-Koblitz formula relating his Gauss sums to the $v$-adic arithmetic gamma values.
	For this reason, we refer to these sums as \textit{arithmetic Gauss sums}.
	On the other hand, there exist other analogs of gamma functions in the function field setting, called geometric \cite[Section 5]{thakur1991gamma} and two-variable \cite[Subsection 9.9]{goss1996basic} gamma functions (see also Definition \ref{defn-v-adic-gamma-definition}).
	This naturally raises the question of whether there are corresponding notions of “geometric” and “two-variable” Gauss sums that align with the frameworks of the theories in \S\ref{subsection-classical-gauss-sums}.
	In this paper, we answer this question affirmatively.
	
	\subsubsection{Geometric Gauss sums}    \label{subsubsection-geometric-gauss-sums}
	
	Let $\nfk \in A_+$ which is relatively prime to $v$.
	We consider the $\nfk$-th cyclotomic function field $K_\nfk$ with ring of integers $\Ocal_{\nfk}$, and choose a prime $\Pfk$ in $K_\nfk$ lying over $v$ with residue degree $\l$.
	Let $\FF_\Pfk := \Ocal_{\nfk}/\Pfk$ be the residue field of $\Pfk$, and denote $C(\FF_{\Pfk})$ the $A$-module structure on $\FF_{\Pfk}$ via the Carlitz $A$-module $C$.
	We let $\omega$ be the $A$-module isomorphism from the $\nfk$-torsions of $C(\FF_{\Pfk})$ to the Carlitz $\nfk$-torsions in $K_\nfk$ which is the inverse of reduction map, called the “geometric” Teichmüller character.
	And let $\psi : \FF_\Pfk \to \Fqdl$ be a fixed $\Fq$-algebra isomorphism.
	Then for any $x \in \nfk^{-1}A$, we define a \textit{geometric Gauss sum} to be
	$$
	\bggs_x
	:= \bggs_x(\Pfk,\psi)
	:= 1 + \sum_{z \in \FF_\Pfk^\times} \omega\left(C_{x(v^\l-1)} (z^{-1})\right)\psi(z),
	$$
	which lies in the compositum $K_{\nfk,d\l} := K_\nfk\Fqdl \sbe \ovl{k}$ of the geometric extension $K_\nfk$ and the constant field extension $k\Fqdl$ over $k$.
	
	To see the analogy between geometric and classical Gauss sums, note that $\omega$ in the geometric case plays the role of “multiplicative character” (although it is in fact additive; cf. \S\ref{subsubsection-a-comparison-of-Gauss-sums}) preserving the $A$-module structures between the residue field and the torsion points of Carlitz module $C(\ovl{k})$.
	This corresponds to the classical case (see \cite[(1.2)]{gk1979gauss}), where the multiplicative characters preserve the $\ZZ$-module structures between the multiplicative group of the residue field and roots of unity in $\ovl{\QQ}^\times$ (i.e., the torsion points of $\ovl{\QQ}^\times$ as a $\ZZ$-module).
	From this point of view, we shall regard $\psi$ as an “additive character”.
	Thus, the inclusion of “plus $1$” in the above definition arises from the natural analogy between addition and multiplication.
	Also note that the usual trace map in the additive character $\psi$ is excluded, as such definitions arise as special cases of the above ones with $\l=1$ (Remark \ref{rem-ggs-with-trace}).
	
	\subsubsection{The Galois group and Gauss sum monomials}    \label{subsubsection-gauss-sum-monomials}
	
	We also define the notion of Gauss sum monomials.
	For a rational number $y \in \QQ$, we define its fractional part $\ang{y}$ as the unique number in $\QQ$ such that $0 \leq \ang{y} <1$ and $y \equiv \ang{y} \pmod{\ZZ}$.
	On the other hand, recall that the $\infty$-adic absolute value on $k$ is given by $|0|_\infty := 0$ and $|f/g|_\infty := q^{\deg f - \deg g}$ for $f,g \in A \setminus\{0\}$.
	Then for each $x \in k$, we define its “$A$-fractional part” $\anginf{x}$ as the unique element in $k$ such that $0 \leq |\anginf{x}|_\infty < 1$ and $x \equiv \anginf{x} \pmod{A}$.
	
	The Galois group $\Gal(K_{\nfk,d\l}/k)$ is canonically isomorphic to $\Gal(K_\nfk/k) \times \Gal(k\Fqdl/k)$.
	We denote $\sigma_{a,s}$ the element in $\Gal(K_{\nfk,d\l}/k)$ extending $\rho_a \in \Gal(K_\nfk/k)$ and $\tau_q^s \in \Gal(k\Fqdl/k)$, where $\rho_a$ corresponds to $a \in (A/\nfk)^\times$ via the Artin map and $\tau_q$ is the $q$-th power Frobenius on the constant field.
	To ease the notation, we also use $\tau_q$ to denote the canonical extension from $\Gal(k\Fqdl/k)$ to $\Gal(K_{\nfk,d\l}/k)$.
	
	For any $y \in (q^{d\l}-1)^{-1} \ZZ$, we write
	$$
	\ang{y} = \sum_{s=0}^{d\l-1} \frac{y_s q^s}{q^{d\l}-1}
	\quad
	(0 \leq y_s < q \text{ for all } s).
	$$
	We consider the integral group ring element $\sum_{s=0}^{d\l-1} y_s \tau_q^s \in \ZZ[\Gal(K_{\nfk,d\l}/k)]$ acting on $\bggs_x$, and define the \textit{geometric Gauss sum monomials}
	$$
	\ggs(x,y)
	:= (\bggs_x)^{\sum_{s=0}^{d\l-1} y_s \tau_q^s}
	= \prod_{s=0}^{d\l-1} (\bggs_x)^{y_s \tau_q^s} \in K_{\nfk,d\l}.
	$$
	In particular, put
	\begin{equation}     \label{eq-intro-gauss-sum-monomial}
		\ggs(x)
		:= \ggs \left(x , \frac{1}{q-1} \right)
		= \prod_{s=0}^{d\l-1} (\bggs_x)^{\tau_q^s} \in K_\nfk.
	\end{equation}
	We also define the \textit{two-variable Gauss sums} and their monomials upon dividing by the arithmetic ones (see \eqref{eq-two-variable-gauss-sum}).
	
	\subsection{The properties of geometric Gauss sums}      \label{subsection-the-properties-of-geometric-gauss-sums}
	
	We establish several fundamental properties of geometric Gauss sums and Gauss sum monomials.
	In the remaining of \S\ref{section-introduction}, we take $\psi: \FF_{\Pfk} \to \Fqdl \sbe \ovl{k}$ to be the usual Teichmüller character.
	We note that any other such $\Fq$-algebra isomorphisms are some $q$-th powers of $\psi$.
	
	\subsubsection{Reflection formula}
	
	The following theorem is viewed as a reflection formula of geometric Gauss sums.
	See \cite[Lemma 6.1(b)]{washington1997introduction}, \cite[(2.1)]{gk1979gauss} for the classical result, and \cite[Theorem II]{thakur1988gauss}, \cite[\nopp 2.1]{thakur1993behaviour} for the arithmetic case.
	
	\begin{thm}[Reflection formula, Theorem \ref{thm-gauss-sum-reflection}]    \label{thm-intro-gauss-sum-reflection}
		For any $x \in \nfk^{-1}A \setminus A$, we have
		$$
		\prod_{\epsilon \in \Fqst} \prod_{s=0}^{d\l-1} (\bggs_x)^{\sigma_{\epsilon,s}}
		=\prod_{\epsilon \in \Fqst} \ggs(\epsilon x)
		= v^\l.
		$$
		In particular, the prime factorization of $\bggs_x$ only involves primes above $v$.
	\end{thm}
	
	\subsubsection{Analog of Stickelberger's theorem}
	
	We also establish the exact prime factorizations of geometric Gauss sums as principal ideals, which is an analog of classical Stickelberger's theorem \cite{stickelberger1890ueber} (see also \cite[Theorem 2.2]{lang1990cyclotomic}).
	See \cite[Theorem IV(1)]{thakur1988gauss} for the analogous result in the arithmetic case.
	
	Let $\Ocal_{\nfk}$ and $\Pfk_{\nfk} := \Pfk$ be defined as in \S\ref{subsubsection-geometric-gauss-sums}.
	We further let $\Ocal_{\nfk,d\l}$ be the ring of integers of $K_{\nfk,d\l}$ and $\Pfk_{\nfk,d\l}$ be the unique prime in $K_{\nfk,d\l}$ above $\Pfk_\nfk$ which is the kernel of $\Ocal_{\nfk,d\l} \simeq \Fqdl \otimes_{\Fq} \Ocal_{\nfk} \to \Fqdl$ given by $\epsilon \otimes \alpha \mapsto \epsilon\psi(\ovl{\alpha})$ where $\epsilon \in \Fqdl$ and $\alpha \in \Ocal_{\nfk}$.
	
	\begin{thm}[Analog of Stickelberger's theorem, Theorem \ref{thm-stickelberger}]
		Given $x \in k$ with $0 < |x|_\infty < 1$, write $x = a_0/\nfk$ where $\deg a_0 < \deg \nfk$ and $(a_0,\nfk)=1$.
		Define
		$$
		\eta_{x,\nfk,d\l} := \sigma_{a_0,\deg \nfk} \cdot \eta_{\nfk,d\l}
		\quad
		\text{where}
		\quad
		\eta_{\nfk,d\l} := \sum_{\substack{a\in A_+ \\ \deg a < \deg \nfk \\ (a,\nfk)=1}}  \sigma_{a,\deg a}^{-1} \in \ZZ[\Gal(K_{\nfk,d\l}/k)].
		$$
		Then $\bggs_x$ has prime factorization
		$$
		\bggs_x \cdot \Ocal_{\nfk,d\l}
		= \Pfk_{\nfk,d\l}^{\eta_{x,\nfk,d\l}}.
		$$
		Consequently, let
		$$
		\eta_{x,\nfk} := \rho_{a_0} \cdot \eta_{\nfk}
		\quad
		\text{where}
		\quad
		\eta_{\nfk} := \sum_{\substack{a\in A_+ \\ \deg a < \deg \nfk \\ (a,\nfk)=1}}  \rho_a^{-1} \in \ZZ[\Gal(K_\nfk/k)].
		$$
		Then $\ggs(x)$ has prime factorization
		$$
		\ggs(x) \cdot \Ocal_{\nfk}
		= \Pfk_{\nfk}^{\eta_{x,\nfk}}.
		$$
	\end{thm}
	
	Thus, the integral group ring elements $\eta_{\nfk}$ and $\eta_{\nfk,d\l}$ are analogs of the classical Stickelberger element corresponding to the cyclotomic function field $K_{\nfk}$ and the composite field $K_{\nfk,d\l}$, respectively.
	It can also be shown that these two elements arise from the $L$-function evaluators corresponding to their respective fields (see \cite[Proposition 15.15]{rosen2002number} for the former).
	Hence, our geometric Gauss sums provide an explicit construction of the Brumer-Stark units for these two abelian extensions over $k$.
	(See Gross' theorem \cite{gross1980annihilation} and the Brumer-Stark conjecture for function fields \cite[Chapter V]{tate1984lesconjectures}, \cite{hayes1985stickelberger}, and \cite[Chapter 15]{rosen2002number}.)
	
	\subsubsection{Analog of Hasse-Davenport product relation}
	
	The following theorem is an analog of Hasse-Davenport product relation \cite{hd1935dienullstellen} (see also \cite[(3.3)]{gk1979gauss}).
	See \cite[Remark 4.6.5(2)]{thakur2004function} for the analogous result in the arithmetic case.
	
	\begin{thm}[Analog of Hasse-Davenport product relation, Theorem \ref{thm-hd-geo-product}]
		Suppose $x \in \nfk^{-1}A$.
		For any $g \in A_{+,h}$ with $(g,v) = 1$, we let $f$ be the order of $v$ modulo $g\nfk$.
		Then for any $y \in (q^{df}-1)^{-1}\ZZ$, we have
		$$
		\prod_{\alpha} \ggsf\left( \frac{x+\alpha}{g},y \right) \bigg/ \ggsf\left( \frac{\alpha}{g},y \right) = \ggsf(x,q^hy)
		$$
		where $\alpha$ runs through a complete residue system modulo $g$.
		In particular, we have
		$$
		\prod_{\alpha} \ggsf\left( \frac{x+\alpha}{g} \right) \bigg/ \ggsf\left( \frac{\alpha}{g} \right) = \ggsf(x).
		$$
	\end{thm}
	
	We mention that by Theorem \ref{thm-intro-gauss-sum-reflection}, the denominator in the second equality is an explicit integral power of $v$ (see Remark \ref{rem-hd-v-power}).
	Also, we refer the readers to \S\ref{subsubsection-hasse-davenport-relations} for other analogs of Hasse-Davenport relations.
	In particular, see Theorem \ref{thm-hd-two-product} for the corresponding result of the two-variable case, which provides an interesting relation between geometric Gauss sums and Thakur's arithmetic Gauss sums in \S\ref{subsubsection-thakurs-arithmetic-gauss-sums}.
	
	\subsubsection{Absolute values and signs at infinity}
	
	There is also a uniform behavior on the $\infty$-adic absolute values of geometric Gauss sums.
	See \cite[Lemma 6.1(c)]{washington1997introduction} for the classical case and \cite[Theorem IV(2)]{thakur1988gauss} for the arithmetic case.
	
	\begin{prop}[$\infty$-adic absolute values, Proposition \ref{prop-absolute-values}]
		The valuation (normalized so that the valuation of $\T$ is $-1$) of the geometric Gauss sums $(\bggs_x)^{\tau_q^s}$ at any infinite place of $K_{\nfk,d\l}$ is $-1/(q-1)$ for all $x \in \nfk^{-1}A \setminus A$ and $0 \leq s \leq d\l-1$.
		In particular, the normalized $\infty$-adic valuation of the Gauss sum monomials $\ggs(x)$ is $-d\l/(q-1)$ for all $x \in \nfk^{-1}A \setminus A$.
	\end{prop}
	
	\begin{rem}[Signs at infinity]
		Contrary to the classical situation, we will give the explicit determination of the signs of geometric Gauss sums and Gauss sum monomials at infinity in Proposition \ref{prop-sign}.
	\end{rem}
	
	\subsubsection{A comparison of Gauss sums}    \label{subsubsection-a-comparison-of-Gauss-sums}
	
	At this point, it is worth indicating one of the main differences between geometric Gauss sums and the other two types of Gauss sums in the existing literature.
	Classically, Gauss sums arise from a multiplicative and an additive character.
	In the arithmetic case \cite{thakur1988gauss} (see also \eqref{eq-ari-gauss-sum}), Thakur makes analogy by preserving the multiplicative character while replacing the additive one with an appropriate $A$-module isomorphism.
	(In fact, one needs not only a multiplicative character but an $\Fq$-algebra isomorphism in order to get non-vanishing Gauss sums; see \cite[Proposition I]{thakur1988gauss} and the remark after that.)
	
	In the geometric case (recall \S\ref{subsubsection-geometric-gauss-sums}), on the other hand, the multiplicative character is replaced by an $A$-module map, and the additive one is replaced by an $\Fq$-algebra isomorphism.
	This transition from multiplication to addition leads us to adopt a completely different approach from both the classical and arithmetic cases.
	In particular, we establish the properties of geometric Gauss sums by first relating them to special $v$-adic geometric and two-variable gamma values, which are analogs of \textit{Gross-Koblitz-Thakur formula}.
	
	\subsection{Gross-Koblitz-Thakur formulas}
	
	Let $k_v$ be the completion of $k$ at $v$ and $A_v$ be its ring of integers.
	Let $\CC_v$ be the completion of a fixed algebraic closure of $k_v$ so that we have the inclusions $K_\nfk \sbe K_{\nfk,d\l} \sbe \ovl{k} \sbe \CC_v$.
	
	\subsubsection{\texorpdfstring{$v$}{v}-adic gamma functions}
	
	For an element $x\in A_v$, we set
	$$
	x^\flat := 
	\begin{cases}
		x, & \text{if } x \in A_v^\times, \\
		1, & \text{if } x\in v A_v.
	\end{cases}
	$$
	In \cite[Section 5]{thakur1991gamma}, Thakur introduced a $v$-adic analog of gamma function $\vgg: A_v \to A_v$, now called \textit{$v$-adic geometric gamma function}, as
	$$
	\vgg(x) := \frac{1}{x^\flat} \prod_{i=0}^{\infty} \left( \prod_{a\in \Ami} \frac{a^\flat}{(x+a)^\flat} \right).
	$$
	This definition was later generalized by Goss at the very end of \cite[Subsection 9.9]{goss1996basic}.
	Specifically, define $\vgg: A_v \times \ZZ_p \to A_v$ by
	$$
	\vgg(x,y+1) := \frac{1}{x^\flat} \prod_{i=0}^\infty \left( \prod_{a\in\Ami} \frac{a^\flat}{(x+a)^\flat} \right)^{y_i}
	$$
	where $y = \sum_{i=0}^\infty y_iq^i$ with $0 \leq y_i <q$ for all $i$.
	One sees that
	$$
	\vgg\left(x , 1-\frac{1}{q-1}\right)
	= \vgg(x).
	$$
	Inspired by Thakur’s modification in the $\infty$-adic case \cite[Section 8]{thakur1991gamma}, we will also consider the \textit{$v$-adic two-variable gamma function} upon dividing by the arithmetic ones defined by Goss \cite[Appendix]{goss1980modular} (Definition \ref{defn-v-adic-gamma-definition}(3)).
	
	\subsubsection{Gross-Koblitz-Thakur formulas}
	
	As mentioned in \S\ref{subsubsection-a-comparison-of-Gauss-sums}, we prove analogs of Gross-Koblitz-Thakur formula for geometric and two-variable gamma functions.
	More precisely, we show that the special $v$-adic gamma monomials can be recognized as the special monomials of geometric Gauss sums in \S\ref{subsubsection-gauss-sum-monomials} as follows.
	(See \cite[Theorem 1.7]{gk1979gauss} for the original Gross-Koblitz formula, and \cite[Theorem VI]{thakur1988gauss}, \eqref{eq-gkt-formula-for-ari} for Thakur's arithmetic analog.)
	
	\begin{thm}[Analogs of Gross-Koblitz-Thakur formula, Theorem \ref{thm-gkt-formula-for-geo}]    \label{thm-intro-gkt-formula}
		For any $x \in \nfk^{-1}A$ and $y \in (q^{d\l}-1)^{-1} \ZZ$, we have
		$$
		\ggs(x)
		= \kappa_1 \cdot
		\prod_{i=0}^{\l-1} \vgg \left( \anginf{v^ix} \right)^{-1}
		$$
		and
		$$
		\ggs (x,y)
		= \kappa_2 \cdot
		\prod_{i=0}^{\l-1} \vgg \left( \anginf{v^ix}, \ang{|v^i|_\infty y} \right)
		$$
		where $\kappa_1$ and $\kappa_2$ are explicitly given algebraic numbers over $k$.
		In particular, $\vgg(a/(v-1))$ and $\vgg(a/(v-1),r/(q^d-1))$ are algebraic over $k$ for all $a \in A$ and $r \in \ZZ$.
	\end{thm}
	
	Theorem \ref{thm-intro-gkt-formula} implies that the product of $v$-adic gamma values along the $\Frob_v$-orbits is algebraic, which aligns with Thakur's theorem for the geometric case \cite[Corollary 8.6.2]{thakur2004function}.
	Thus, Theorem \ref{thm-intro-gkt-formula} not only gives a “Gauss sum” explanation of Thakur's result, but also generalizes it to the two-variable case (when taking Thakur's arithmetic analog into account; see Theorem \ref{thm-gkt-formula-for-two}).
	
	\subsubsection{Strategy of proof}
	
	To prove Theorem \ref{thm-intro-gkt-formula}, one key step involves interpreting the geometric Gauss sums as the “reduction of the twisted Coleman functions” evaluated at a particular point (Theorem \ref{thm-coleman-function-and-gauss-sum}; see also Remark \ref{rem-ggs-and-coleman}).
	These functions were first introduced by Anderson in \cite{anderson1992twodimensional} using the soliton machinery and later reconstructed by Anderson-Brownawell-Papanikolas in \cite[\nopp 6.3.5]{abp2004determination} more directly.
	They also play a pivotal role in the period interpretations of $\infty$-adic gamma values in positive characteristic, as demonstrated in \cite{sinha1997periods}, \cite{abp2004determination}, and \cite{wei2022algebraic}.
	
	In order to prove Theorem \ref{thm-coleman-function-and-gauss-sum}, we make use of a “duality” identity in \cite[Theorem 5.4.4]{abp2004determination} (see also Theorem \ref{thm-restatement-of-abp-5.4.4}).
	Since the characters defining the geometric Gauss sums are both $\Fq$-linear, we are naturally led to the comparison between three different pairings: the residue pairing, the Poonen pairing, and the trace pairing (see \S\ref{subsection-pairing-comparisons}).
	Theorem \ref{thm-coleman-function-and-gauss-sum} then follows from the compatibility among these pairings (Theorem \ref{thm-pairing-summary}).
	
	\begin{rem}
		Using Theorem \ref{thm-intro-gkt-formula}, together with Thakur’s theory of $v$-adic gamma functions \cite{thakur1991gamma} (e.g., their standard functional equations), we are able to establish the properties of geometric Gauss sums presented in \S\ref{subsection-the-properties-of-geometric-gauss-sums}.
		We also use the compatibility among pairings to determine their signs at infinity in \S\ref{subsection-the-behavior-of-geometric-gauss-sums-at-infinity}.
	\end{rem}
	
	\begin{rem}
		Using his original definition of Coleman functions, Anderson \cite{anderson1992twodimensional} constructed “Jacobi sum” Hecke characters for $K_\nfk$, which parallel the classical result due to Weil \cite{weil1952jacobi}.
		Under this setting, he also developed the reflection formula, Stickelberger's theorem, and absolute values.
		Thus, by Theorem \ref{thm-coleman-function-and-gauss-sum}, we are able to reinterpret Anderson's Jacobi sums in terms of geometric Gauss sum monomials in \S\ref{subsubsection-gauss-sum-monomials} (cf. \cite[\nopp 3.4]{anderson1992twodimensional} and \eqref{eq-intro-gauss-sum-monomial}), which would extend his theory to the composite field $K_{\nfk,d\l}$.
	\end{rem}
	
	\subsection*{Acknowledgments}
	
	The author would like to express sincere gratitude to Fu-Tsun Wei for his exceptional mentorship, insightful discussions, and constant support throughout the development of this work.
	Appreciation is also extended to Jing Yu for providing valuable comments and feedback.
	Special thanks to the National Science and Technology Council for its financial support over the past few years, and to the National Center for Theoretical Sciences for organizing workshops and conferences that have enriched his understanding of mathematics and contributed to his academic growth.
	
	\section{Preliminaries}    \label{section-preliminaries}
	
	Let $A := \Fq[\T]$ be the polynomial ring in the variable $\T$ over a finite field $\Fq$ of $q$ elements, where $q$ is a power of a prime number $p$, and $k := \Fq(\T)$ be the field of fractions of $A$ with a fixed algebraic closure $\ovl{k}$.
	Let $A_+$ be the set of all monic polynomials in $A$ and $\Ami$ be its subset consisting of monic polynomials of degree $i\geq 0$.
	
	\subsection{Carlitz and adjoint Carlitz modules}
	
	In this section, we recall the notions of Carlitz and adjoint Carlitz modules over an arbitrary $A$-field.
	See \cite[Chapter 3]{goss1996basic} for more information.
	
	\subsubsection{Carlitz modules}     \label{subsubsection-carlitz-modules}
	
	Let $(L,\iota)$ be an $A$-field where $\iota: A \to L$ is the structural $\Fq$-algebra homomorphism.
	Let $\tau$ be the $q$-th power Frobenius element in the $\Fq$-linear endomorphism ring $\End_{\Fq}(\GG_a)$ of the additive group $\GG_a$ of $L$.
	This ring is isomorphic to the twisted polynomial ring $L\{\tau\}$ with the usual addition rule and the multiplication given by $\tau \alpha = \alpha^q \tau$ for all $\alpha \in L$.
	It is also isomorphic to the ring of $\Fq$-linear polynomials with the usual addition rule and the multiplication given by composition.
	Under this identification, each element $\alpha\tau^n \in L\{\tau\}$ corresponds to the $\Fq$-linear polynomial $\alpha z^{q^n}$ for all $\alpha\in L$.
	
	The \textit{Carlitz module over $L$} is the $\Fq$-algebra homomorphism $C: A \to L\{\tau\}$ given by $C_\T := \iota(\T) + \tau$. 
	Since $A$ is generated freely by $\T$ over $\Fq$, the Carlitz module is determined completely by $C_\T$.
	In the language of $\Fq$-linear polynomials, we have $C_\T(z) = \iota(\T)z + z^q$, called a \textit{Carlitz polynomial}.
	More generally, for each $a \in A$, a straightforward calculation shows that the Carlitz polynomial $C_a(z)$ is
	\begin{equation}    \label{eq-carlitz-polynomial}
		C_a(z) = \sum_{i=0}^{\deg a} c_{a,i} z^{q^i} \in \iota(A)[z] \sbe L[z]
	\end{equation}
	where $c_{a,0} = \iota(a)$ and $c_{a,\deg a}$ is the leading coefficient of $a$.
	
	Via the structure map $\iota$, there is a natural $A$-module action on $L$ given by the usual multiplication in $L$.
	Now via $C$, we obtain a new $A$-module structure on $L$ given by $a \cdot \alpha := C_a(\alpha)$ for all $a\in A$ and $\alpha \in L$.
	The latter structure will be denoted as $C(L)$ to distinguish the usual $A$-module structure on $L$.
	Very often, we identify the Carlitz module over $L$ with $C(L)$.
	
	\subsubsection{Adjoint Carlitz modules}     \label{subsubsection-adjoint-carlitz-modules}
	
	Suppose now the $A$-field $L$ is perfect, so that $\tau$ induces an automorphism of $L$.
	We let $L\{\tau^{-1}\}$ be the twisted polynomial ring in $\tau^{-1}$ over $L$ with the usual addition rule and the multiplication given by $\tau^{-1} \alpha = \alpha^{1/q} \tau^{-1}$ for all $\alpha \in L$.
	The \textit{adjoint Carlitz module over $L$} is defined by the $\Fq$-algebra homomorphism $C^*: A \to L\{\tau^{-1}\}$ where $C^*_\T := \iota(\T) + \tau^{-1}$.
	Similar to the Carlitz module, the adjoint Carlitz module is determined completely by $C^*_\T$.
	For each $a\in A$, the adjoint of the Carlitz polynomial \eqref{eq-carlitz-polynomial} is
	\begin{equation}      \label{eq-adjoint-carlitz-polynomial}
		C^*_a(z) := \sum_{i=0}^{\deg a} c_{a,i}^{q^{-i}} z^{q^{-i}}.
	\end{equation}
	Via $C^*$, we obtain a new $A$-module structure $C^*(L)$ on $L$, which will also be identified with the adjoint Carlitz module over $L$.
	
	\subsection{Carlitz's torsion points and cyclotomic function fields}
	
	In this section, we review some fundamental properties of cyclotomic function fields.
	We refer the readers to \cite{hayes1974explicit}, \cite[Chapter 12]{rosen2002number}, or \cite[Section 7.1]{papikian2023drinfeld}.
	
	\subsubsection{Cyclotomic function fields}    \label{subsubsection-cyclotomic-function-fields}
	
	Consider the Carlitz module $C(\ovl{k})$ with the natural structure map $\iota: A \inj k \inj \ovl{k}$.
	We fix $\nfk \in A_+$ and let $\Lambda_\nfk$ be the $\nfk$-torsion points of $C(\ovl{k})$.
	In other words, $\Lambda_\nfk$ consists of roots of the Carlitz polynomial $C_\nfk(z)$ in $\ovl{k}$ (recall \eqref{eq-carlitz-polynomial}).
	We let $K_\nfk := k(\Lambda_\nfk)$, called the \textit{$\nfk$-th cyclotomic function field}, and $\Ocal_\nfk$ be the integral closure of $A$ in $K_\nfk$.
	Then the following properties are known.
	
	\begin{itemize}
		\item There is an $A$-module isomorphism $\Lambda_{\nfk} \simeq A/\nfk$.
		In particular, $\Lambda_{\nfk}$ has a generator $\lambda_\nfk \in \Lambda_\nfk$, called a primitive $\nfk$-th root of $C$.
		For $a\in A$, $C_a(\lambda_\nfk)$ is primitive if and only if $(a,\nfk) = 1$.
		
		\item The $\nfk$-th cyclotomic polynomial $C_\nfk^\star (z) := \prod (z-\lambda)$, where $\lambda$ runs through all primitive $\nfk$-th roots of $C$, has coefficients in $A$.
		
		\item $K_\nfk/k$ is Galois and $\Gal(K_\nfk/k) \simeq (A/\nfk A)^\times$ where each $a \in (A/\nfk A)^\times$ corresponds to the automorphism $\rho_a \in \Gal(K_\nfk/k)$ such that $\rho_a(\lambda_\nfk) = C_a(\lambda_\nfk)$ (the Artin map).
		
		\item The ring $\Ocal_\nfk = A[\lambda_\nfk]$.
		
		\item For any prime $v \in A_+$, the polynomial $C_v(z)/z$ is Eisenstein at $v$.
		Moreover, $v$ is ramified in $K_\nfk$ if and only if $v \mid \nfk$.
		
		\item If $v \nmid \nfk$, let $\l$ be the order of $v$ in $(A/\nfk A)^\times$.
		Then $v$ splits into $[K_\nfk : k]/\l$ primes in $K_\nfk$, each with residue degree $\l$.
		
		\item The infinite prime $\infty$ in $k$ splits into $[K_\nfk : k]/(q-1)$ primes in $K_\nfk$ with the decomposition group $\{\rho_\epsilon \mid \epsilon \in \Fqst\} \sbe \Gal(K_\nfk/k)$ which is isomorphic to $\Fqst$.
	\end{itemize}
	
	\subsubsection{Torsions of the Carlitz modules}
	
	We have the following proposition.
	
	\begin{prop}     \label{prop-reduction-of-lambda}
		Let $\nfk \in A_+$ and $v \in A_+$ be an irreducible polynomial not dividing $\nfk$.
		Let $\Pfk$ be a prime in $K_\nfk$ above $v$ of degree $\l$ with residue field $\FF_{\Pfk} := \Ocal_\nfk/\Pfk$.
		Then we have the $A$-module isomorphisms $A/(v^\l-1) \simeq \Lambda_{v^\l-1} \simeq C(\FF_{\Pfk})$, where the second map is given by reduction modulo $\Pfk$.
		In particular, we have the isomorphism $\Lambda_{\nfk} \simeq C(\FF_{\Pfk})[\nfk]$.
	\end{prop}
	
	\begin{proof}
		The first isomorphism is a general fact mentioned earlier.
		For the second, from
		$$
		C_\nfk(z) = z\prod_{0\neq \lambda \in \Lambda_\nfk} (z - \lambda),
		$$
		we take derivative and substitute $z=0$, which results in
		$$
		\nfk = \prod_{0\neq \lambda \in \Lambda_\nfk} (-\lambda).
		$$
		Since the reduction of $\nfk$ is non-zero in $\FF_{\Pfk}$, we obtain an injection $\Lambda_{\nfk} \to C(\FF_{\Pfk})[\nfk] \sbe C(\FF_\Pfk)$.
		In particular, we have $\Lambda_{v^\l-1} \simeq C(\FF_{\Pfk})$ by counting cardinality.
		(Here, we identify the residue field of $v$ in $K_{v^\l-1}$ with $\FF_{\Pfk}$.)
	\end{proof}
	
	\subsubsection{Torsions of the adjoint Carlitz modules}       \label{subsubsection-torsions-of-the-adjoint-carlitz-modules}
	
	Now, consider the adjoint Carlitz module $C^*(\ovl{k})$.
	For $\nfk \in A_+$, we let $\Lambda_\nfk^*$ be the $\nfk$-torsion points of $C^*(\ovl{k})$.
	In other words, $\Lambda_\nfk^*$ consists of roots of $C_\nfk^*(z)$ in $\ovl{k}$
	(or more precisely, roots of the polynomial $C_\nfk^*(z)^{q^{\deg \nfk}} \in A[z]$; recall \eqref{eq-adjoint-carlitz-polynomial}).
	Then we have the following theorem.
	
	\begin{thm}      \label{thm-goss-1.7.11}
		The roots of $C_\nfk(z)$ and $C_\nfk^*(z)$ generate the same field extension over $k$. In other words, $K_\nfk = k(\Lambda_\nfk) = k(\Lambda_\nfk^*)$.
	\end{thm}
	
	\begin{proof}
		See \cite[Theorem 1.7.11]{goss1996basic}.
	\end{proof}
	
	Note that the leading coefficient of the polynomial $C_\nfk^*(z)^{q^{\deg\nfk}} \in A[z]$ is $\nfk^{q^{\deg \nfk}}$.
	So every $\lambda^* \in \Lambda_\nfk^* \sbe K_\nfk$ is integral at $v$ for each irreducible $v \in A_+$ not dividing $\nfk$.
	Hence, we may consider a statement similar to Proposition \ref{prop-reduction-of-lambda}.
	
	\begin{prop}     \label{prop-reduction-of-lambda*}
		Let $\nfk \in A_+$ and $v \in A_+$ be an irreducible polynomial not dividing $\nfk$.
		Let $\Pfk$ be a prime in $K_\nfk$ above $v$ of degree $\l$ with residue field $\FF_{\Pfk} := \Ocal_\nfk/\Pfk$ (which is perfect).
		Then we have the $A$-module isomorphisms $A/(v^\l-1) \simeq \Lambda^*_{v^\l-1} \simeq C^*(\FF_{\Pfk})$.
	\end{prop}
	
	\begin{proof}
		A similar argument to that of Proposition \ref{prop-reduction-of-lambda} applies in the adjoint setting.
		The first isomorphism is also a general property.
		For the second, one considers
		$$
		(C_\nfk^*(z) / \nfk)^{q^{\deg\nfk}}
		= z\prod_{0 \neq \lambda^* \in \Lambda_\nfk^*} (z - \lambda^*).
		$$
		We omit the details.
	\end{proof}
	
	\subsection{A duality identity}       \label{subsection-a-duality-identity}
	
	In this section, we recall an important duality result in \cite{abp2004determination}.
	
	\subsubsection{Dual families}
	
	Let $\Res: k \to \Fq$ be the usual residue map for the variable $\T$.
	For $\nfk \in A_+$, define the residue pairing $\angres{\cdot,\cdot}: A/\nfk \times A/\nfk \to \Fq$ by $\angres{a,b} := \Res(ab/\nfk)$.
	Since with respect to the ordered bases $\{1,\T,\ldots,\T^{\deg\nfk-1}\}$ and $\{\T^{\deg\nfk-1},\T^{\deg\nfk-2},\ldots,1\}$, the matrix representation of $\angres{\cdot,\cdot}$ is lower triangular with $1$'s along the diagonal, the residue pairing is perfect.
	We call dual bases $\{a_i\}_{i=1}^{\deg\nfk}, \{b_j\}_{j=1}^{\deg\nfk}$ ($a_i,b_j\in A$) of this pairing \textit{$\nfk$-dual families}.
	
	\subsubsection{An Ore's result}    \label{subsubsection-an-ores-result}
	
	Fix any $\nfk$-dual families $\{a_i\}_{i=1}^{\deg\nfk}, \{b_j\}_{j=1}^{\deg\nfk}$.
	We choose an $A$-module isomorphism from $A/\nfk$ to $\Lambda_\nfk$, and suppose the element $1\in A/\nfk$ is mapped to $\lambda \in \Lambda_\nfk$.
	Let $\lambda_j := C_{b_j}(\lambda)$ be the corresponding image of $b_j$.
	Since the given $A$-module isomorphism is $\Fq$-linear, $\{\lambda_j\}_{j=1}^{\deg\nfk}$ becomes an $\Fq$-basis of $\Lambda_\nfk$.
	An explicit formula given by Ore then says that one may express an $\Fq$-basis of $\Lambda_\nfk^*$ in terms of $\{\lambda_j\}_{j=1}^{\deg\nfk}$ and the Moore determinant.
	For $x_1,\ldots,x_n \in \ovl{k}$, define $\Delta(x_1,\ldots,x_n) := \det_{1 \leq i,j \leq n} x_j^{q^{i-1}}$.
	
	\begin{thm}[Ore]      \label{thm-ore}
		For each $1\leq i \leq \deg\nfk$, set $\Delta_i := \Delta(\lambda_1,\ldots,\lambda_{i-1},\lambda_{i+1},\ldots,\lambda_{\deg\nfk})$ and $\Delta := \Delta(\lambda_1,\ldots,\lambda_{\deg\nfk}) \neq 0$. Let
		$$
		\lambda_i^* := (-1)^{\deg\nfk+i} \left(\frac{\Delta_i}{\Delta}\right)^q.
		$$
		Then $\{\lambda_i^*\}_{i=1}^{\deg\nfk}$ forms an $\Fq$-basis of $\Lambda_\nfk^*$.
	\end{thm}
	
	\begin{proof}
		See \cite[\nopp 8]{ore1933special} or \cite[Theorem 1.7.13]{goss1996basic}.
	\end{proof}
	
	\begin{rem}
		In fact, these $\lambda_i^*$ are the dual bases of $\lambda_j$ with respect to the “Poonen pairing” (with a minus sign modification).
		We will justify this claim in Proposition \ref{prop-lambdai*-and-poonen-pairing}.
	\end{rem}
	
	\subsubsection{A duality identity}
	
	For each integer $N\geq 0$, define $\Psi_N(z) \in k[z]$ so that
	\begin{equation}      \label{eq-Psi}
		1 + \Psi_N(z)
		= \prod_{a\in A_{+,N}} \left(1+\frac{z}{a}\right),
	\end{equation}
	where the product runs through all monic polynomials of degree $N$ in $A$.
	Then we have the following identity between $\{\lambda_i^*\}_{i=1}^{\deg\nfk}$ and $\{\lambda_j\}_{j=1}^{\deg\nfk}$, which is in fact an “algebraic” reformulation of \cite[Theorem 5.4.4]{abp2004determination}.
	See \S\ref{subsubsection-a-duality-identity-between-torsions-ofe-and-e^*} (in particular, \eqref{eq-the-equation-of-the-main-goal}) for the comparison with the original statement.
	
	\begin{thm}  \label{thm-restatement-of-abp-5.4.4}
		Fix $\nfk \in A_+$ of positive degree and $\nfk$-dual families 
		$$
		\{a_i\}_{i=1}^{\deg\nfk},
		\quad
		\{b_j\}_{j=1}^{\deg\nfk}.
		$$
		Then for $a_0 \in A$ with $\deg a_0 < \deg \nfk$, we have  \\
		(1)
		$$
		\sum_{i=1}^{\deg \nfk} (\lambda_i^*)^{q^N} C_{a_0}(\lambda_i) = -\Psi_N(a_0/\nfk)
		$$
		for all integers $N\geq 0$.   \\
		(2) Moreover, if $a_0 \in A_+$, we have
		$$
		\sum_{i=1}^{\deg \nfk} \lambda_i^* C_{a_0}(\lambda_i)^{q^{\deg \nfk - \deg a_0}} = 1.
		$$
	\end{thm}
	
	\section{Geometric Gauss sums}
	
	In this section, we introduce the notion of geometric Gauss sums and establish some basic properties.
	We also derive an equivalent scalar product expression which will be a key identity in \S\ref{section-gross-koblitz-thakur-formulas-and-their-applications}.
	
	\subsection{Definition}
	
	\subsubsection{Definition of geometric Gauss sums}    \label{subsubsection-definition-of-gauss-sums}
	
	Let $\nfk \in A_+$ and $v \in A_{+,d}$ be irreducible and relatively prime to $\nfk$ with $d>0$.
	Let $\Lambda_\nfk$ be the $\nfk$-torsion points of the Carlitz module $C(\ovl{k})$, $K_\nfk := k(\Lambda_\nfk)$ be the $\nfk$-th cyclotomic function field and $\Ocal_\nfk$ be the integral closure of $A$ in $K_\nfk$.
	Fix a prime $\Pfk$ in $\Ocal_\nfk$ above $v$ with residue degree $\l$, and put $\FF_\Pfk := \Ocal_\nfk/\Pfk$.
	We let $\omega: C(\FF_{\Pfk})[\nfk] \to \Lambda_\nfk$ be the $A$-module isomorphism from the $\nfk$-torsions of $C(\FF_{\Pfk})$ to the Carlitz $\nfk$-torsions in $K_\nfk$ which is the inverse of reduction map (Proposition \ref{prop-reduction-of-lambda}).
	We regard $\omega$ as the “geometric” Teichmüller character (see \eqref{eq-geo-teichmüller}).
	Fix an $\Fq$-algebra isomorphism $\psi: \FF_\Pfk \to \Fqdl$.
	
	\begin{defn}    \label{defn-ggs}
		For any $x \in \nfk^{-1}A$, we define a \textit{geometric Gauss sum} to be
		$$
		\bggs_x
		:= \bggs_x (\Pfk,\psi)
		:= 1 + \sum_{z \in \FF_\Pfk^\times} \omega\left(C_{x(v^\l-1)}(z^{-1})\right)\psi(z) \in K_{\nfk,d\l},
		$$
		where $K_{\nfk,d\l} \sbe \ovl{k}$ is the compositum of the geometric extension $K_\nfk$ and the constant field extension $k\Fqdl$ over $k$.
	\end{defn}
	
	We note that $\bggs_{x} = \bggs_{x+a}$ for all $a \in A$.
	Thus, the geometric Gauss sum is well-defined modulo $A$.
	Also, note that $\bggs_x = 1$ when $x \in A$.
	
	\begin{eg}      \label{eg-ggs}
		Take $v = \T$ and $\nfk = v-1$.
		Then for any $x = \epsilon/\nfk \in \nfk^{-1}A$ where $\epsilon \in \Fq$, we have
		$$
		\bggs_x
		= 1 - \epsilon \omega(1)
		= 1 - \epsilon (1-\T)^{1/(q-1)}
		$$
		where $\omega(1) \in \Lambda_\nfk$ is some $(q-1)$-st root of $1-\T$.
	\end{eg}
	
	\subsubsection{Compatibility}
	
	The geometric Gauss sums satisfy a compatibility condition.
	Suppose we view $x \in \nfk^{-1}A$ as an element in $(\nfk')^{-1}A$, where $\nfk' \in A_+$ is a multiple of $\nfk$ which is also relatively prime to $v$, then the corresponding geometric Gauss sums remain the same.
	To justify this claim, we use the following identity.
	
	\begin{lem}    \label{lem-descending-ggs}
		Let $r$ be a prime power and $n\in\NN$.
		Suppose $h$ is a function on $\FF_{r^n}$ with values in a ring containing $\FF_{r^n}$ and satisfies $h(0) = 0$, then
		$$
		\sum_{z \in \FF_{r^n}^\times} h\left( \Tr_{\FF_{r^n}/\FF_r} (z^{-1}) \right) z^{r^j}
		=  \sum_{z \in \FF_r^\times} h(z^{-1}) z^{r^j}
		$$
		for all $j \geq 0$.
	\end{lem}
	
	\begin{proof}
		This is equivalent to \cite[Lemma I]{thakur1988gauss}.
	\end{proof}
	
	Now, let $\l'$ be the order of $v$ modulo $\nfk'$.
	Then necessarily we have $\l$ divides $\l'$.
	Assume $\l' = m\l$ for some natural number $m$.
	We write $\Pfk$ (resp. $\Pfk'$) as the chosen prime in $K_{\nfk}$ (resp. $K_{\nfk'}$) above $v$ with residue field $\FF_{\Pfk}$ (resp. $\FF_{\Pfk'}$).
	Note that $m$ is the residue degree of $\Pfk'$ over $\Pfk$.
	
	\begin{prop}    \label{prop-compatibility-of-ggs}
		Setting as above, and consider $\omega$ and $\psi$ as functions defined on $C(\FF_{\Pfk'})$ and $\FF_{\Pfk'}$, respectively.
		Then we have
		$$
		\sum_{z \in \FF_{\Pfk}^\times} \omega\left(C_{x(v^\l-1)}(z^{-1})\right)\psi(z)
		= \sum_{z \in \FF_{\Pfk'}^\times} \omega\left(C_{x(v^{\l'}-1)}(z^{-1})\right)\psi(z).
		$$
	\end{prop}
	
	\begin{proof}
		Consider the right-hand side of the proposition.
		For each $z \in \FF_{\Pfk'}^\times$, by linearity and the fact that $C_v \equiv \tau^d \pmod{v}$ sends $z$ to $z^{q^d}$ (recall \S\ref{subsubsection-cyclotomic-function-fields}), we have
		$$
		\omega\left(C_{x(v^{\l'}-1)}(z^{-1})\right)
		= \sum_{i=0}^{m-1} \omega\left(C_{x(v^\l-1)v^{i\l}}(z^{-1})\right)
		= \omega\left(C_{x(v^\l-1)} (\Tr_{\FF_{\Pfk'}/\FF_{\Pfk}}(z^{-1}))\right).
		$$
		The result now follows from Lemma \ref{lem-descending-ggs}.
	\end{proof}
	
	\begin{rem}    \label{rem-ggs-with-trace}
		In Definition \ref{defn-ggs}, we can alternatively adopt the definition by including the usual trace map in the additive character $\psi$ and omitting the “plus $1$”.
		That is, define
		$$
		\widetilde{\Gcal}_{x,\l}^{\textnormal{geo}}
		:= -\sum_{z \in \FF_{\Pfk}^\times} \omega\left(C_{x(v^\l-1)}(z^{-1})\right)\psi \left( \Tr_{\FF_{\Pfk}/\FF_v} (z) \right)
		$$
		where $\FF_v := A/v$.
		Let $\tau_q$ be the $q$-th power Frobenius map on the constant field.
		Then one may express
		\begin{equation}    \label{eq-gg-tilde-and-gg}
			\widetilde{\Gcal}_{x,\l}^{\textnormal{geo}}
			= \sum_{i=0}^{\l-1} (1-\bggs_x)^{\tau_q^{di}}.
		\end{equation}
		Hence, the geometric Gauss sums $\bggs_x$ are viewed as refinements of $\widetilde{\Gcal}_{x,\l}^{\textnormal{geo}}$.
		Also, one sees similarly by change of variables and Lemma \ref{lem-descending-ggs} that
		$$
		\widetilde{\Gcal}_{x,\l}^{\textnormal{geo}}
		= -\sum_{z \in \FF_{\Pfk}^\times} \omega\left(C_{x(v^\l-1)} (\Tr_{\FF_{\Pfk}/\FF_v}(z^{-1}))\right) \psi(z)
		= -\sum_{z \in \FF_v^\times} \omega\left(C_{x(v^\l-1)}(z^{-1})\right) \psi(z)
		= \widetilde{\Gcal}_{x',1}^{\textnormal{geo}}
		$$
		where $x' \in (v-1)^{-1}A$ satisfying $x(v^\l-1) \equiv x'(v-1) \pmod{v-1}$.
		So including the trace map will always reduce the definition to the $\l=1$ case.
	\end{rem}
	
	\subsection{Basic properties}
	
	\subsubsection{An additive Hasse-Davenport lifting relation}    \label{subsubsection-an-additive-hasse-davenport-lifting-relation}
	
	From Lemma \ref{lem-descending-ggs} (or \eqref{eq-gg-tilde-and-gg} and the fact that $\bggs_x$ is fixed by $\tau_q^{d\l}$), one deduces the following “additive” analog of Hasse-Davenport lifting relation (cf. Proposition \ref{prop-compatibility-of-ggs}).
	$$
	m \cdot \widetilde{\Gcal}_{x,\l}^{\textnormal{geo}}
	= \widetilde{\Gcal}_{x,m\l}^{\textnormal{geo}}
	\quad
	\text{for all}
	\quad
	m \in \NN.
	$$
	This deviation from the classical Gauss sums \cite{hd1935dienullstellen}, where $m$ appears in the exponent, is also a natural phenomenon of the addition/multiplication comparison.
	As we will see in \S\ref{subsubsection-hasse-davenport-relations}, the original geometric Gauss sums (i.e., including the “plus $1$”) turn out to satisfy the multiplicative relations in parallel with the classical ones.
	
	\subsubsection{Fourier analysis}
	
	The geometric Gauss sums can be interpreted as the Fourier coefficients of “geometric characters” on $\FF_{\Pfk}^\times$.
	Let $f(z) := \omega(C_{x(v^\l-1)}(z))$.
	Then by the $\Fq$-linearity of $f$ and Fourier inversion formula, one sees that
	$$
	f(z) = \sum_{i=0}^{d\l-1} \left( (1 - \bggs_x) \psi(z) \right)^{\tau_q^i}.
	$$
	Hence, we have
	$$
	\Tr_{\FF_\Pfk/\Fq}(z) - f(z)
	= \sum_{i=0}^{d\l-1} \left( \bggs_x \psi(z) \right)^{\tau_q^i}.
	$$
	
	\subsubsection{Galois actions}      \label{subsubsection-galois-actions}
	
	Note that the Galois group $\Gal(K_{\nfk,d\l}/k)$ is canonically isomorphic to $\Gal(K_\nfk/k) \times \Gal(k\Fqdl/k)$, which is further isomorphic to $(A/\nfk)^\times \times \ZZ/d\l\ZZ$ (the first component is by Artin map; recall \S\ref{subsubsection-cyclotomic-function-fields}).
	We let $\rho_a \in \Gal(K_\nfk/k)$ be the element corresponding to $a \in (A/\nfk)^\times$ and $\tau_q \in \Gal(k\Fqdl/k)$ be the $q$-th power Frobenius on the constant field as similarly defined in Remark \ref{rem-ggs-with-trace}.
	Finally, we let $\sigma_{a,s}$ be the element in $\Gal(K_{\nfk,d\l}/k)$ corresponding to $(\rho_a,\tau_q^s)$ in $\Gal(K_\nfk/k) \times \Gal(k\Fqdl/k)$.
	
	\begin{prop}       \label{prop-galois-actions}
		(1) $(\bggs_x)^{\sigma_{a,0}}
		= \bggs_{ax}$ for all $a \in (A/\nfk)^\times$.      \\
		(2) $(\bggs_x)^{\sigma_{1,-d}} = \bggs_{vx}$.
	\end{prop}
	
	\begin{proof}
		(1) follows from the commutativity of $A$-actions $\rho_a \circ \omega = \omega \circ C_a$.
		For (2), by change of variables and using the fact that $C_v \equiv \tau^d \pmod{v}$, we have
		$$
		(\bggs_x)^{\sigma_{1,-d}}
		= 1 + \sum_{z \in \FF_\Pfk^\times} \omega\left(C_{x(v^\l-1)} (z^{-q^d})\right) \psi(y)
		= 1 + \sum_{z \in \FF_\Pfk^\times} \omega\left(C_{vx(v^\l-1)} (z^{-1})\right) \psi(y)
		= \bggs_{vx}.
		$$
	\end{proof}
	
	\begin{rem}   \label{rem-fixed-field}
		Note that $v$ splits into $[K_\nfk:k]/\l \cdot d\l = [K_\nfk:k]d$ primes in $K_{\nfk,d\l}$.
		So the decomposition group $D$ of $v$ in $K_{\nfk,d\l}$ (note $K_{\nfk,d\l}/k$ is abelian) contains $\l$ elements and is generated by the Artin automorphism $\sigma_{v,d} \in D$.
		Thus, by Proposition \ref{prop-galois-actions}(2), we see that $\bggs_x$ lies in the fixed field $K_{\nfk,d\l}^D$ of $D$.
		Moreover, it also implies
		$$
		\left(\prod_{i=0}^{\l-1} \bggs_{v^ix} \right)^{\sigma_{1,d}} = \prod_{i=0}^{\l-1} \bggs_{v^ix}
		\quad
		\text{and}
		\quad
		\prod_{s=0}^{d-1} \prod_{i=0}^{\l-1} (\bggs_x)^{\sigma_{v^i,s}}
		= \prod_{s=0}^{d\l-1} (\bggs_x)^{\sigma_{1,s}}.
		$$
	\end{rem}
	
	\subsection{A scalar product expression}     \label{subsection-a-scalar-product-expression}
	
	Put $\mfk := v^\l-1$ and let $\Pfk_\mfk$ be a prime in $K_\mfk$ above $\Pfk_\nfk := \Pfk$ fixed in \S\ref{subsubsection-definition-of-gauss-sums} with residue field $\FF_{\Pfk_\mfk}$.
	Since $\FF_{\Pfk_\nfk}$ and $\FF_{\Pfk_\mfk}$ are isomorphic, we may identify
	\begin{equation}     \label{eq-scalar-product-observation}
		\bggs_x
		= 1 + \sum_{z \in \FF_{\Pfk_\nfk}^\times} \omega\left(C_{x(v^\l-1)}(z^{-1})\right)\psi(z)   \\
		= 1 + \sum_{z \in \FF_{\Pfk_\mfk}^\times} \omega\left(C_{\mfk x} (z^{-1})\right)\psi(z).
	\end{equation}
	The following theorem will be a key to \S\ref{section-gross-koblitz-thakur-formulas-and-their-applications}.
	
	\begin{thm}   \label{thm-coleman-function-and-gauss-sum}
		Fix $\mfk$-dual families $\{a_i\}_{i=1}^{d\l}, \{b_j\}_{j=1}^{d\l}$ and let $\{\lambda_i^*\}_{i=1}^{d\l}, \{\lambda_j\}_{j=1}^{d\l}$ be defined accordingly as in \S\ref{subsubsection-an-ores-result}.
		Then we have
		$$
		\bggs_x
		= 1 - \sum_{i=1}^{d\l} C_{\mfk x}(\lambda_i) \psi(\ovl{\lambda}{}^*_i)
		$$
		where we take the reduction of $\lambda_i^*$ modulo $\Pfk_\mfk$.
		(Recall \S\ref{subsubsection-torsions-of-the-adjoint-carlitz-modules} that each $\lambda_i^*$ is integral at $v$.)
	\end{thm}
	
	\begin{proof}
		From Theorem \ref{thm-restatement-of-abp-5.4.4}(2), we have
		$$
		\sum_{i=1}^{d\l} \psi(\ovl{\lambda}{}^*_i) \psi(\ovl{\lambda}_i) = 1
		$$
		where we take reductions modulo $\Pfk_\mfk$.
		Using \eqref{eq-scalar-product-observation} and applying change of variables $z \ovl{\lambda}_i \mapsto z$, we have
		\begin{align*}
			\bggs_x
			&= 1 + \sum_{z \in \FF_{\Pfk_\mfk}^\times} \omega\left(C_{\mfk x} (z^{-1})\right)\psi(z) \sum_{i=1}^{d\l} \psi(\ovl{\lambda}{}^*_i) \psi(\ovl{\lambda}_i)    \\
			&= 1 + \sum_{z \in \FF_{\Pfk_\mfk}^\times} \sum_{i=1}^{d\l} \omega\left(C_{\mfk x} (z^{-1}\ovl{\lambda}_i)\right) \psi(z\ovl{\lambda}{}^*_i).
		\end{align*}
		
		Let $X := (\ovl{\lambda}_1, \ldots, \ovl{\lambda}_{d\l})$ and $Y := (\ovl{\lambda}{}^*_1, \ldots, \ovl{\lambda}{}^*_{d\l})$ be as row vectors. 
		From Propositions \ref{prop-reduction-of-lambda} and \ref{prop-reduction-of-lambda*}, the entries of $X$ and $Y$ are both $\Fq$-bases of the finite field $\FF_{\Pfk_\mfk}$.
		In \S\ref{subsection-pairing-comparisons} (see Theorem \ref{thm-pairing-summary}), we will show that in fact, they are dual bases with respect to the trace pairing $\angtr{\cdot,\cdot}: \FF_{\Pfk_\mfk} \times \FF_{\Pfk_\mfk} \to \Fq$ defined by $\angtr{a,b} := \Tr_{\FF_{\Pfk}/\Fq}(ab)$.
		Since both $\omega \circ C_{\mfk x}$ and $\psi$ are $\Fq$-linear, it follows that the representing matrices of the actions of $z^{-1}$ on $\FF_{\Pfk_\mfk}$ with respect to $X$ and $z$ on $\FF_{\Pfk_\mfk}$ with respect to $Y$ are mutually inverse.
		Hence, we obtain
		$$
		\bggs_x
		= 1 + (q^{d\l} - 1) \sum_{i=1}^{d\l} \omega\left(C_{\mfk x} (\ovl{\lambda}_i)\right) \psi(\ovl{\lambda}{}^*_i)
		= 1 - \sum_{i=1}^{d\l} C_{\mfk x}(\lambda_i) \psi(\ovl{\lambda}{}^*_i).
		$$
		This completes the proof.
	\end{proof}
	
	\begin{eg}
		Consider Example \ref{eg-ggs}.
		Suppose $v = \T, \nfk = v-1$, and $x = \epsilon/\nfk \in \nfk^{-1}A$ for some $\epsilon \in \Fq$.
		In this case, $d = \l = 1$ and $\mfk = \nfk = v-1$.
		One sees by Theorem \ref{thm-ore} that
		$$
		\lambda_1^* = \lambda_1^{-q}.
		$$
		Since $\dim_{\Fq} \Lambda_{\nfk} = 1$ (see \S\ref{subsubsection-cyclotomic-function-fields}) and $\omega(1), \lambda_1 \in \Lambda_\nfk \setminus \{0\}$, there exists $\eta \in \Fqst$ such that $\omega(1) = \eta \lambda_1$.
		That is, the reduction of $\eta \lambda_1$ modulo $\Pfk_\mfk$ is the identity element $1 \in \FF_{\Pfk_\mfk} \simeq \Fq$.
		Then by Theorem \ref{thm-coleman-function-and-gauss-sum},
		$$
		\bggs_x
		= 1 - \epsilon\lambda_1 \psi(\ovl{\lambda}{}^*_1)
		= 1 - \epsilon \lambda_1 \eta
		= 1- \epsilon \omega(1),
		$$
		which aligns with Example \ref{eg-ggs}.
	\end{eg}
	
	\begin{rem}   \label{rem-ggs-and-coleman}
		Let $t,z$ be two independent variables.
		For each $a \in A$, we let
		$$
		C_a(t,z) := C_a(z)|_{\T = t} \in \Fq[t,z]
		\quad
		\text{and}
		\quad
		C_a^\star (t,z) := C_a^\star(z)|_{\T = t} \in \Fq[t,z],
		$$
		where $C_a(z)$ is the Carlitz polynomial and $C_a^\star (z)$ is the $a$-th cyclotomic polynomial.
		Let $O_a := \Fq[t,z]/(C_a^\star (t,z))$.
		Then with notations as in Theorem \ref{thm-coleman-function-and-gauss-sum}, for each $x \in \mfk^{-1}A$, the \textit{Coleman function} is defined as (see \cite[\nopp 6.3.5]{abp2004determination})
		$$
		g_x(t,z) := 1 - \sum_{i=1}^{d\l} C_{\mfk xb_i}(t,z) (\lambda_i^*)^{1/q}
		\in \ovl{k} \otimes_{\Fq} O_\mfk.
		$$
		Note we have $g_x = g_{x+a}$ for all $a \in A$.
		Moreover, one sees by definition that $C_{\mfk xb_i}(\T,\lambda) = C_{\mfk xb_i}(\lambda) = C_{\mfk x}(C_{b_i}(\lambda))
		= C_{\mfk x}(\lambda_i)$.
		Hence, by Theorem \ref{thm-coleman-function-and-gauss-sum}, we may regard the geometric Gauss sums as the “reduction at $v$” of the twisted Coleman functions specialized at $(\T,\lambda)$.
	\end{rem}
	
	Other fundamental properties of geometric Gauss sums, including their absolute values, analogs of Hasse-Davenport relations, and Stickelberger's theorem will be established in \S\ref{subsection-more-on-geometric-gauss-sums}.
	We will also determine their signs at infinity in \S\ref{subsection-the-behavior-of-geometric-gauss-sums-at-infinity}.
	
	\section{\texorpdfstring{$v$}{v}-adic gamma functions}     \label{section-v-adic-gamma-functions}
	
	In this section, we recall three $v$-adic gamma functions over $k$, and derive the standard functional equations for the two-variable case.
	Fix an irreducible $v\in A_{+,d}$ with $d>0$.
	We let $k_v$ be the completion of $k$ at $v$ with ring of integers $A_v$, and $\CC_v$ be the completion of a fixed algebraic closure of $k_v$ containing $\ovl{k}$.
	
	\subsection{Definition}
	
	\subsubsection{Three \texorpdfstring{$v$}{v}-adic gamma functions}
	
	A $p$-adic integer $y \in \ZZ_p$, when written as $y = \sum y_i q^i$, is always assumed that $0\leq y_i < q$.
	In particular, $y_i = 0$ for all $i < 0$.
	For an element $x\in A_v$, we set
	$$
	x^\flat := 
	\begin{cases}
		x, & \text{if } x \in A_v^\times, \\
		1, & \text{if } x\in v A_v.
	\end{cases}
	$$
	
	Recall the following three analogs of $v$-adic gamma and factorial functions.
	
	\begin{defn}    \label{defn-v-adic-gamma-definition}
		(1) \textit{$v$-adic arithmetic gamma function} (\cite[Appendix]{goss1980modular}): Define $\vag: \ZZ_p \to A_v$ by
		$$
		\vag (y) := \vaf(y-1)
		\quad
		\text{where}
		\quad
		\vaf(y) := \prod_{i=0}^{\infty} \left( -\prod_{a\in \Ami} a^\flat \right)^{y_i}.
		$$
		(2) \textit{$v$-adic geometric gamma function} (\cite[Section 5]{thakur1991gamma}): Define $\vgg: A_v \to A_v$ by
		$$
		\vgg(x) := \frac{1}{x^\flat} \vgf(x)
		\quad
		\text{where}
		\quad
		\vgf(x) := \prod_{i=0}^{\infty} \left( \prod_{a\in \Ami} \frac{a^\flat}{(x+a)^\flat} \right).
		$$
		(3) \textit{$v$-adic two-variable gamma function} (\cite[Subsection 9.9]{goss1996basic}): Define $\vtg: A_v \times \ZZ_p \to A_v$ by
		$$
		\vtg(x,y)
		:= \frac{\vgg(x,y)}{\vag(y)}
		= \frac{1}{x^\flat} \vtf(x,y-1)
		\quad
		\text{and}
		\quad
		\vtf(x,y)
		:= \frac{\vgf(x,y)}{\vaf(y)}
		$$
		where
		$$
		\vgg(x,y) := \frac{1}{x^\flat} \vgf(x,y-1)
		\quad
		\text{and}
		\quad
		\vgf(x,y) := \prod_{i=0}^\infty \left( \prod_{a\in\Ami} \frac{a^\flat}{(x+a)^\flat} \right)^{y_i}.
		$$
	\end{defn}
	
	Actually, Goss' definition in the two-variable case is the $\vgg(x,y)$ above, which generalizes the geometric case as one has
	$$
	\vgg\left(x , 1-\frac{1}{q-1}\right)
	= \vgg(x).
	$$
	Inspired by Thakur's modification of its $\infty$-adic counterpart \cite[Section 8]{thakur1991gamma}, we define the $v$-adic two-variable gamma function with the arithmetic one involved.
	In particular, we have
	$$
	\vtg\left(x , 1-\frac{1}{q-1}\right) 
	= \epsilon \vgg(x),
	\quad
	\text{where}
	\quad
	\epsilon = \vag\left(1-\frac{1}{q-1}\right)^{-1}
	$$
	is a root of unity \cite[Theorem 4.4]{thakur1991gamma}.
	Thus, our $v$-adic two-variable gamma function generalizes both the arithmetic and geometric cases.
	
	\subsubsection{The convergence of the two-variable case}
	
	The convergence of the two-variable case is justified by the following lemma.
	
	\begin{lem}
		Let $x,y\in A$ and $n,m\in\NN$ with $x\equiv y \pmod{v^r}$ and $n\equiv m\pmod{q^s}$ for large $r,s \in \NN$.
		Then we have
		$$
		\vgf(x,n) \equiv \vgf(x,m) \pmod{v^{\floor{s/d}}}
		\quad
		\text{and}
		\quad
		\vgf(x,m) \equiv \vgf(y,m) \pmod{v^r}.
		$$
	\end{lem}
	
	\begin{proof}
		The second congruence equation follows from the observation that for each $a\in\Ami$, $(x+a)^\flat$ and $(y+a)^\flat$ are either simultaneously $1$ or $x+a$ and $y+a$.
		For the first one, we write $n=\sum n_iq^i$ and $m=\sum m_iq^i$ in $q$-adic expansions.
		By assumption we have $n_i = m_i$ for $i < s$.
		So it's sufficient to show that
		$$
		\prod_{a\in\Ami} \frac{a^\flat}{(x+a)^\flat} \equiv 1 \pmod{v^{\floor{s/d}}}
		$$
		for all $i\geq s$.
		For the numerator, write $i=ed+j$ where $0\leq j < d$.
		And for each $a\in\Ami$, write $a=f\cdot v^e+g$ where $\deg f=j$ and $\deg g<ed$.
		Note that $v\nmid a$ if and only if $v\nmid g$, and as $a$ runs through $\Ami$, we get $q^i/q^{ed} = q^j$ copies of complete residue system modulo $v^e$.
		Hence, we see by Wilson's theorem that
		$$
		\prod_{a\in\Ami} a^\flat
		= \prod_{\substack{a\in\Ami \\ v\nmid a}} a
		\equiv \left(\prod_{g \in (A/v^eA)^\times} g \right)^{q^j} 
		\equiv -1 \pmod{v^e = v^{\floor{i/d}}}.
		$$
		This shows that the numerator is congruent to $-1$ modulo $v^{\floor{s/d}}$.
		A similar argument shows that this is also true for the denominator.
		Writing $x+a = f\cdot v^e+g$, we observe that there are $q^j$ possibilities of $f$, each corresponding to a complete residue system modulo $v^e$.
		Hence, the result follows.
	\end{proof}
	
	\subsection{Functional equations}
	
	Natural functional equations of the arithmetic and geometric gammas, such as their respective reflection and multiplication formulas, were well established by Thakur in \cite{thakur1991gamma}.
	Using the same idea and argument, one may easily deduce the corresponding results for the two-variable case.
	(We note that both $\vaf(\cdot)$ and $\vgf(x,\cdot)$ fit into the framework of \cite[Section 2]{thakur1991gamma}.)
	
	\begin{prop}
		Let $y,y' \in \ZZ_p$ with $y = \sum y_iq^i$ and $y' = \sum y_i'q^i$ in $q$-adic expansions.
		If $y_i+y_i' < q$ for all $i$ (i.e., $y+y'$ has no carry over base $q$), we have
		$$
		\frac{\vtg(x,1+y+y')}{\vtg(x,1+y)\vtg(x,1+y')}
		= x^\flat.
		$$
	\end{prop}
	
	\subsubsection{Reflection formula}
	
	\begin{thm}
		For $(x,y) \in A_v \times \ZZ_p$, we have  \\
		(1)
		$$
		\vtg(x,y) \vtg(x,1-y)
		= (-1)^{d-1} (x^{\flat})^{q-3} \vgg(x)^{q-1}.
		$$
		(2)
		$$
		\prod_{\epsilon\in \Fqst}
		\vtg(\epsilon x,y) \vtg(\epsilon x,1-y)
		= \left( \frac{1}{x^{\flat}} \right)^{q-1}.
		$$
	\end{thm}
	
	\subsubsection{Multiplication formula}
	
	\begin{thm}
		For $(x,y) \in A_v \times \ZZ_p$ and $n\in\NN$ with $(n,q)=1$, we have
		$$
		\prod_{i=0}^{n-1} \vtg\left(x,\frac{y+i}{n}\right)
		= \left(\frac{1}{x^\flat}\right)^{n-1}
		\left( \frac{\vgf(x,-1)}{\vaf(-1)} \right)^{(n-1)/2}
		\vtg(x,y)
		$$
		where the square root is taken to be $\vgf(x,-1/2)/\vaf(-1/2)$ when $n$ is even.
	\end{thm}
	
	We define $\partial_v: \ZZ_p \to \ZZ_p$ as $\partial_v(\sum_{i=0}^{\infty} y_i q^i) := \sum_{i=0}^\infty y_{i+d}q^i$, which shifts the $q$-adic digits of a $p$-adic integer $d$ steps left.
	We also define the sign function $\sgn$ on $A$ which sends a non-zero polynomial to its leading coefficient, and put $\sgn(0) := 0$.
	
	\begin{thm}    \label{thm-multiplication-formula}
		Let $x_0 \in A$ be the unique element such that $x_0 \equiv x \pmod{v}$ and $\deg x_0 < d$.
		Then for any $g\in A_{+,h}$ with $(g,v) = 1$ and $y = \sum y_i q^i \in \ZZ_p$, we have
		$$
		\prod_{\alpha} \vtf\left(\frac{x+\alpha}{g},y\right) = \vtf(x,q^h y) g^{-\delta_x y_{i'-h} + m} (g^{q^d-1})^{\partial_v(q^hy)}.
		$$
		Equivalently,
		$$
		\prod_{\alpha} \vgf\left(\frac{x+\alpha}{g},y\right)
		= \vgf(x,q^h y) \frac{1}{\vaf(q^hy)} \vaf(y)^{q^h}
		g^{-\delta_x y_{i'-h} + m} (g^{q^d-1})^{\partial_v(q^hy)}.
		$$
		Here, $\alpha$ runs through a complete residue system modulo $g$, $i':=\deg x_0$, $\delta_x := 1$ if $h \leq i' <d$ and $\sgn(x_0) = -1$, and $\delta_x := 0$ otherwise, $m:=\sum_{i=h}^{d-1} y_{i-h}q^i$ is a non-negative integer, and $(g^{q^d-1})^{\partial_v(q^hy)} \in A_v$ is a one-unit.
	\end{thm}
	
	\begin{proof}
		Starting with $\vgf(\cdot,\cdot)$ (Definition \ref{defn-v-adic-gamma-definition}(3)), we have
		\begin{multline}   \label{eq-multiplication-formula-1}
			\prod_\alpha \vgf\left(\frac{x+\alpha}{g},y\right)
			= \lim_{N\to\infty}
			\left( \prod_{i=0}^N \left( \prod_{a\in\Ami} \prod_{\alpha} \frac{1}{(ga+\alpha+x)^\flat} \right)^{y_i} \right)   \\
			\left( \prod_{i=0}^N \prod_{a\in\Ami} \prod_\alpha (a^\flat)^{y_i} \right)
			\left( \prod_{i=0}^N \left( \prod_{a\in\Ami} \prod_{\substack{\alpha \\ v\nmid ga+\alpha+x}} g \right)^{y_i} \right). 
		\end{multline}
		In the first term, we see from the division algorithm that for each such $a$ and $\alpha$, $ga+\alpha$ corresponds to a unique $b \in A_+$ with $h \leq \deg b\leq N+h$, and vice versa.
		So by change of variables, the first two terms of \eqref{eq-multiplication-formula-1} are
		\begin{align*}
			&\left( \prod_{i=0}^N \left( \prod_{a\in\Ami} \prod_{\alpha} \frac{1}{(ga+\alpha+x)^\flat} \right)^{y_i} \right) 
			\left( \prod_{i=0}^N \prod_{a\in\Ami} \prod_\alpha (a^\flat)^{y_i} \right)    \\
			={} &\left( \prod_{i=h}^{N+h} \left( \prod_{b\in\Ami} \frac{b^\flat}{(b+x)^\flat} \right)^{y_{i-h}} \right)
			\left( \prod_{i=h}^{N+h} \left( \prod_{b\in\Ami} \frac{1}{b^\flat} \right)^{y_{i-h}} \right)
			\left( \prod_{i=0}^N \prod_{a\in\Ami} (a^\flat)^{y_i} \right)^{q^h}
		\end{align*}
		Note that the first term converges to $\vgf(x,q^h y)$.
		And for the latter two, by multiplying $-1$ to each of their terms, they converge to $1/\vaf(q^h y)$ and $\vaf(y)^{q^h}$, respectively.
		Recall by Definition \ref{defn-v-adic-gamma-definition}(3), $\vtf(x,y) = \vgf(x,y)/\vaf(y)$.
		Thus, \eqref{eq-multiplication-formula-1} becomes
		$$
		\prod_{\alpha} \vtf\left(\frac{x+\alpha}{g},y\right) = \vtf(x,q^h y)  \lim_{N\to\infty} \prod_{i=0}^N \left( \prod_{a\in\Ami} \prod_{\substack{\alpha \\ v\nmid ga+\alpha+x}} g \right)^{y_i}.
		$$
		
		It remains to compute the limit.
		Using the same change of variables, we see that
		$$
		\lim_{N\to\infty} \prod_{i=0}^N \left( \prod_{a\in\Ami} \prod_{\substack{\alpha \\ v\nmid ga+\alpha+x}} g \right)^{y_i}
		= \lim_{N\to\infty} \prod_{i=h}^{N+h} \left( \prod_{b\in\Ami} \prod_{ v \nmid b+x} g \right)^{y_{i-h}}
		= \lim_{N\to\infty} g^{\sum_{i=h}^{N+h} k_iy_{i-h}}
		$$
		where
		$$
		k_i =
		\begin{cases}
			(q^d-1)q^{i-d}, & \text{if } i \geq d,  \\
			q^i-1, & \text{if } h \leq i = i' < d  \text{ and } \sgn(x_0)=-1, \\
			q^i, & \text{otherwise},
		\end{cases}
		$$
		which counts the cardinality of the set $\{b \in\Ami \mid v \nmid b+x\}$.
		By the definitions of $\delta_x,i'$ and $m$, we have
		$$
		\sum_{i=h}^{N+h} k_iy_{i-h}
		= -\delta_x y_{i'-h} + m + (q^d-1)\sum_{i=d}^{N+h} y_{i-h} q^{i-d}.
		$$
		Finally, since $g^{q^d-1}$ is a one-unit in $A_v$, by non-Archimedean analysis, the exponential function with base $g^{q^d-1}$ interpolates to $\ZZ_p$ (see \cite[Theorem 32.4]{schikhof1985ultrametric}).
		Thus, we have
		$$
		\lim_{N \to \infty} g^{(q^d-1) \sum_{i=d}^{N+h} y_{i-h} q^{i-d}}
		= (g^{q^d-1})^{\partial_v(q^hy)}
		$$
		which is a one-unit.
	\end{proof}
	
	\subsubsection{monomial relations}
	
	We conclude this section by presenting some monomial relations among special $v$-adic gamma values.
	They are all immediate consequences of the functional equations of $\vtg(\cdot,\cdot)$ we have seen in this section.
	For $\kappa_1,\kappa_2 \in \CC_v^\times$, we write $\kappa_1 \sim \kappa_2$ if $\kappa_1/\kappa_2 \in \ovl{k}^\times$.
	
	\begin{cor}
		Let $x\in (k \cap A_v) \setminus A$ and $y\in (\QQ \cap \ZZ_p) \setminus \ZZ$.  \\
		(1) For any $a\in A$ and $n \in \ZZ$,
		$$
		\vtg(x+a,y+n)
		\sim
		\vtg(x,y).
		$$
		(2)
		$$
		\vtg(x,y) \vtg(x,1-y)
		\sim \vgg(x)^{q-1}
		\quad
		\text{and}
		\quad
		\prod_{\epsilon\in \Fqst}
		\vtg(\epsilon x,y) \vtg(\epsilon x,1-y)
		\sim 1.
		$$
		(3) Suppose $y \in (q^t-1)^{-1}\ZZ$.
		Write the fractional part of $-y$ as $\sum_{s=0}^{t-1} r_sq^s/(q^t -1)$ where $0 \leq r_s < q$ for all $s$.
		Then
		$$
		\vtg(x,y)
		\sim
		\prod_{s=0}^{t-1} \vtg \left(x , 1-\frac{q^s}{q^t-1}\right)^{r_s}.
		$$
		(4) For any $n \in \NN$ with $(n,q) = 1$,
		$$
		\prod_{i=0}^{n-1} \vtg\left( x,\frac{y+i}{n} \right)
		\sim
		\vtg(x,y)\vgg(x)^{(n-1)(q-1)/2}.
		$$
		(5) For any $g \in A_{+,h}$ with $(g,v) = 1$,
		$$
		\prod_{\substack{\alpha \in A \\ \deg \alpha < \deg g}} \vtg \left(\frac{x+\alpha}{g},y\right)
		\sim
		\vtg(x,q^h y). 
		$$
	\end{cor}
	
	\section{Gross-Koblitz-Thakur formulas and their applications}      \label{section-gross-koblitz-thakur-formulas-and-their-applications}
	
	In this section, we prove Gross-Koblitz-Thakur formulas for $v$-adic geometric and two-variable gamma functions, and use them to establish further properties of geometric Gauss sums.
	These include the absolute values, analogs of Hasse-Davenport relations, and an analog of Stickelberger's theorem on prime factorizations.
	
	\subsection{Gross-Koblitz-Thakur formulas for \texorpdfstring{$v$}{v}-adic gamma functions}
		
	\subsubsection{Arithmetic and geometric Teichmüller characters}     \label{subsubsection-arithmetic-and-geometric-characters}
	
	As in \S\ref{subsubsection-definition-of-gauss-sums}, we let $\nfk \in A_+$ which is relatively prime to $v$ and $\mfk := v^\l-1$ where $\l$ is the order of $v$ modulo $\nfk$.
	Recall that we assumed $\Pfk$ is any prime in $K_\nfk$ above $v$, and $\psi: \FF_\Pfk \to \Fqdl$ is any $\Fq$-algebra isomorphism.
	Now, we choose $\Pfk$ so that the completion of $K_\nfk$ at $\Pfk$ is contained in $\CC_v$, as fixed in \S\ref{section-v-adic-gamma-functions} (with ring of integers $\Ocal_{\nfk,\Pfk}$), and $\psi$ as the Teichmüller embedding from $\FF_\Pfk = \Ocal_\nfk/\Pfk$ into $\Fqdl \sbe \Ocal_{\nfk,\Pfk}$.
	We mention that any other such $\Fq$-algebra isomorphisms are some $q$-th powers of the Teichmüller embedding.
	
	We can regard $\psi: \FF_{\Pfk} \to \Fqdl$ as the “arithmetic” Teichmüller character, which satisfies
	\begin{equation}   \label{eq-ari-teichmüller}
		\psi(\ovl{\alpha}) = \lim_{N \to \infty} \alpha^{q^{Nd\l}}
		\quad
		\text{for all}
		\quad
		\alpha \in \Ocal_{\nfk,\Pfk}.
	\end{equation}
	Correspondingly, for the “geometric” Teichmüller character $\omega: C(\FF_{\Pfk}) \to \Lambda_{\mfk}$ (inverse of the reduction map), we have
	\begin{equation}    \label{eq-geo-teichmüller}
		\omega(\ovl{\alpha}) = \lim_{N \to \infty} C_{v^{N\l}} (\alpha)
		\quad
		\text{for all}
		\quad
		\alpha \in \Ocal_{\nfk,\Pfk}.
	\end{equation}
	
	\subsubsection{Geometric Gauss sum monomials}
	
	For each $y \in \QQ$, we define its fractional part $\ang{y}$ as the unique number such that $0 \leq \ang{y} <1$ and $y \equiv \ang{y} \pmod{\ZZ}$.
	Similarly, for each $x \in k$, we define its $A$-fractional part $\anginf{x}$ as the unique element in $k$ such that $0 \leq |\anginf{x}|_\infty < 1$ and $x \equiv \anginf{x} \pmod{A}$.
	(Here, $|\cdot|_\infty$ is the usual $\infty$-adic absolute value on $k$.)
	
	For any $x \in \nfk^{-1} A$ and $y \in (q^{d\l}-1)^{-1} \ZZ$, we write
	$$
	\ang{y} = \sum_{s=0}^{d\l-1} \frac{y_s q^s}{q^{d\l}-1}
	\quad
	(0 \leq y_s < q \text{ for all } s),
	$$
	and define the special monomials of geometric Gauss sums by the action of the group ring element $\sum_{s=0}^{d\l-1} y_s \tau_q^s \in \ZZ[\Gal(K_{\nfk,d\l}/k)]$ on $\bggs_x$ (recall \S\ref{subsubsection-galois-actions} that $\tau_q$ denotes the $q$-th power Frobenius on the constant field, which will also be extended canonically to $K_{\nfk,d\l}$).
	That is, define
	\begin{equation}    \label{eq-gauss-sum-monomial}
		\ggs(x,y) := \prod_{s=0}^{d\l-1} (\bggs_x)^{y_s \tau_q^s}
		\quad
		\text{and}
		\quad
		\ggs(x)
		:= \ggs \left(x , \frac{1}{q-1} \right)
		:= \prod_{s=0}^{d\l-1} (\bggs_x)^{\tau_q^s}.
	\end{equation}
	
	\subsubsection{A special case}
	
	All of the results in \S\ref{section-gross-koblitz-thakur-formulas-and-their-applications} are based on the following calculation.
	
	\begin{thm}      \label{thm-first-gkt-formula}
		Suppose $x \in \nfk^{-1} A$ and $y \in (q^{d\l}-1)^{-1} \ZZ$ are given by
		$$
		\anginf{x} = \sum_{i=0}^{\l-1} \frac{x_i v^i}{v^\l-1}
		\quad
		(\deg x_i < d \text{ for all } i),
		\quad
		\ang{y} = \frac{q^s}{q^{d\l}-1}
		\quad
		(0 \leq s < d\l).
		$$
		Then we have
		$$
		\ggs(x,y)
		= \delta_x^{(s)} \cdot
		\prod_{i=0}^{\l-1} \vgf \left( \anginf{v^ix}, -\ang{|v^i|_\infty y} \right)^{-1}
		$$
		where $\delta_x^{(s)}$ is defined as follows:
		Let $0 \leq e \leq \l-1$ be the unique integer such that $ed \leq s < (e+1)d$ and $s' := s - ed$ be the unique integer such that $s' \equiv s \pmod{d}$ and $0 \leq s' < d$.
		Then $\delta_x^{(s)} := v \anginf{v^{\l-e-1}x}$ if $x_e \in A_{+,s'}$ and $1$ otherwise.
	\end{thm}
	
	\begin{proof}
		Since all the quantities are unchanged when replacing $(x,y)$ by $(\anginf{x},\ang{y})$, we may assume $\anginf{x} = x$ and $\ang{y} = y$.
		For each $0 \leq i \leq \l-1$, set
		$$
		\ang{|v^i|_\infty y}
		= \ang{q^{di}y}
		= \ang{\frac{q^{s+di}}{q^{d\l}-1}}
		=: \frac{q^{s_i}}{q^{d\l}-1}
		$$
		where $s_i \equiv s+di \pmod{d\l}$ with $0 \leq s_i < d\l$ (the index is considered modulo $\l$).
		Then by the definition of $\vgf(\cdot,\cdot)$ (Definition \ref{defn-v-adic-gamma-definition}(3)), we have
		\begin{equation}   \label{eq-first-gkt-1}
			\prod_{i=0}^{\l-1} \vgf \left( \anginf{v^ix}, -\ang{q^{di}y} \right)
			= \lim_{N\to\infty}  \prod_{i=0}^{\l-1} \prod_{j=0}^N  \prod_{a \in A_{+,s_i+jd\l}} \frac{a^\flat}{\left(\anginf{v^ix}+a\right)^\flat}.       
		\end{equation}
		Put $x_\l := x_0$ and note that $\anginf{v^ix} \equiv -x_{\l-i} \pmod{v}$ in $A_v$ for all $0\leq i \leq \l-1$.
		So
		\begin{equation}      \label{eq-gkt-analyze-den}
			v \mid \anginf{v^ix} + a
			\iff a \equiv x_{\l-i} \pmod{v}.
		\end{equation}
		
		Recall the definition of $\Psi_N(z)$ in  \eqref{eq-Psi} and consider the right-hand side of \eqref{eq-first-gkt-1}.
		For the $s_i = s'$ (which means $i \equiv \l-e \pmod{\l}$) and $j=0$ term, using \eqref{eq-gkt-analyze-den} and the property $0 \leq s' < d$, we have
		$$
		\prod_{a \in A_{+,s'}} \frac{a^\flat}{\left(\anginf{v^{\l-e} x}+a\right)^\flat}
		= \delta_x^{(s)}
		\prod_{a \in A_{+,s'}} 	\frac{a}{\anginf{v^{\l-e} x}+a}
		= \delta_x^{(s)} \left(1+ \Psi_{s'}\left(\anginf{v^{\l-e} x}\right)\right)^{-1}.
		$$
		For $s_i = s'$ and each $j \geq 1$, one investigates the $v$-multiple parts and obtains
		\begin{align*}
			&\prod_{j=1}^N \prod_{a \in A_{+,s'+jd\l}} \frac{a^\flat}{\left(\anginf{v^{\l-e}x}+a\right)^\flat}   \\
			={} &\prod_{j=1}^N \frac{ \prod_{a\in A_{+,s'+jd\l}} a \Big/ \prod_{a\in A_{+,s'+jd\l-d}} va }{ \prod_{a\in A_{+,s'+jd\l}} \left(\anginf{v^{\l-e}x}+a\right) \Big/ \prod_{a\in A_{+,s'+jd\l-d}} v\left(\anginf{v^{\l-e-1}x} + a\right) }      \\
			={} &\prod_{j=1}^N \frac{1+\Psi_{s'+jd\l-d} \left(\anginf{v^{\l-e-1}x}\right)}{1+\Psi_{s'+jd\l} \left(\anginf{v^{\l-e}x}\right)}.
		\end{align*}
		For $s_i \neq s'$ (which means $i \not\equiv \l-e \pmod{\l}$), we note that $d\leq s_i < d\l$.
		Therefore, similar computation yields
		$$
		\prod_{j=0}^N \prod_{a \in A_{+,s_i+jd\l}} \frac{a^\flat}{\left(\anginf{v^ix}+a\right)^\flat}
		= \prod_{j=0}^N \frac{1+\Psi_{s_i+jd\l-d} \left(\anginf{v^{\l+i-1}x}\right)}{1+\Psi_{s_i+jd\l} \left(\anginf{v^ix}\right)}.
		$$
		Thus, the right-hand side of \eqref{eq-first-gkt-1} is seen to be
		\begin{multline*}
			\delta_x^{(s)}
			\lim_{N\to\infty} \left( \frac{1}{1+ \Psi_{s'} \left(\anginf{v^{\l-e}x}\right)} \prod_{j=1}^N \frac{1+\Psi_{s'+jd\l-d} \left(\anginf{v^{\l-e-1}x}\right)}{1+\Psi_{s'+jd\l} \left(\anginf{v^{\l-e}x}\right)} \right)
			\\
			\left( \prod_{\substack{i=0 \\ i \not\equiv \l-e \bmod \l}}^{\l-1} \prod_{j=0}^N \frac{1+\Psi_{s_i+jd\l-d} \left(\anginf{v^{\l+i-1}x}\right)}{1+\Psi_{s_i+jd\l} \left(\anginf{v^ix}\right)} \right).
		\end{multline*}
		A careful inspection shows that except for the term
		$$
		\left(1+\Psi_{s_{\l-e-1}+Nd\l} \left(\anginf{v^{\l-e-1}x}\right) \right)^{-1}
		= \left(1+\Psi_{s+d(\l-e-1)+Nd\l} \left(\anginf{v^{\l-e-1}x}\right) \right)^{-1},
		$$
		the denominator of the former product will cancel out the numerator of the latter one.
		Finally, by applying Theorem \ref{thm-restatement-of-abp-5.4.4}(1), \eqref{eq-ari-teichmüller}, Theorem \ref{thm-coleman-function-and-gauss-sum}, Proposition \ref{prop-galois-actions}(2), and \eqref{eq-gauss-sum-monomial} sequentially, we obtain
		\begin{align*}
			&\lim_{N \to \infty} \left(1+\Psi_{s+d(\l-e-1)+Nd\l} \left(\anginf{v^{\l-e-1}x}\right) \right)^{-1}   \\
			={} &\lim_{N \to \infty} \left( 1 - \sum_{i=1}^{d\l} (\lambda_i^*)^{q^{s+d(\l-e-1)+Nd\l}} C_{\mfk\anginf{v^{\l-e-1}x}}(\lambda_i) \right)^{-1}    \\
			={} &\left( 1 - \sum_{i=1}^{d\l} \psi(\ovl{\lambda}{}^*_i)^{q^{s+d(\l-e-1)}} C_{\mfk\anginf{v^{\l-e-1}x}}(\lambda_i) \right)^{-1}   \\
			={} &\left(\bggs_{v^{\l-e-1}x}\right)^{-\tau_q^{s+d(\l-e-1)}}
			= (\bggs_x)^{-\tau_q^s}
			= \ggs (x,y)^{-1}.
		\end{align*}
		This completes the proof.
	\end{proof}
	
	\begin{rem}
		In the above proof, the telescoping product of the special $v$-adic gamma values and the algebraicity of the $v$-adic limit were first demonstrated by Thakur in the geometric case \cite[Section 8.6]{thakur2004function}.
		In there, the limit was interpreted as the Coleman functions evaluated at particular point.
		Our Theorem \ref{thm-first-gkt-formula} generalizes his result to the two-variable case and explains the limit in terms of geometric Gauss sums (recall Remark \ref{rem-ggs-and-coleman}).
	\end{rem}
	
	\subsubsection{The geometric case}
	
	Note that the formula in Theorem \ref{thm-first-gkt-formula} involves not a product but a single geometric Gauss sum.
	Hence, we may use it to obtain analogous formulas for the geometric and two-variable gamma functions.
	
	\begin{thm}[Gross-Koblitz-Thakur formula for geometric gamma function]      \label{thm-gkt-formula-for-geo}
		For any $x \in \nfk^{-1} A$ and $y \in (q^{d\l}-1)^{-1} \ZZ$, write
		$$
		\anginf{x} = \sum_{i=0}^{\l-1} \frac{x_i v^i}{v^\l-1}
		\quad
		(\deg x_i < d \text{ for all } i),
		\quad
		\ang{y} = \sum_{s=0}^{d\l-1} \frac{y_s q^s}{q^{d\l}-1}
		\quad
		(0 \leq y_s < q \text{ for all } s).
		$$
		Then we have
		$$
		\ggs (x,y)
		= \prod_{i=0}^{\l-1}
		\delta_{x,i}^{y_{di+\deg x_i}} \cdot
		\prod_{i=0}^{\l-1} \vgf \left( \anginf{v^ix}, -\ang{|v^i|_\infty y} \right)^{-1}
		$$
		where $\delta_{x,i} := v\anginf{v^{\l-i-1}x}$ if $x_i \in A_+$ and $1$ otherwise.
		In particular, we have
		$$
		\ggs (x)
		= \prod_{i=0}^{\l-1} \delta_{x,i}
		\cdot
		\prod_{i=0}^{\l-1} \vgf \left( \anginf{v^ix} \right)^{-1}.
		$$
	\end{thm}
	
	\begin{proof}
		We may once again assume $\anginf{x} = x$ and $\ang{y} = y$.
		Note that the second assertion follows immediately from the first by taking $y = 1/(q-1)$ (recall \eqref{eq-gauss-sum-monomial} and Definition \ref{defn-v-adic-gamma-definition}).
		On the other hand, by Theorem \ref{thm-first-gkt-formula}, one sees that
		\begin{align*}
			&\prod_{i=0}^{\l-1} \vgf \left( \anginf{v^ix}, -\ang{q^{di} y} \right)
			= \prod_{s=0}^{d\l-1} \prod_{i=0}^{\l-1} \vgf \left( \anginf{v^ix}, -\ang{\frac{q^{s+di}}{q^{d\l}-1}} \right)^{y_s}      \\
			={} &\prod_{s=0}^{d\l-1} \left( \delta_x^{(s)} \ggs\left(x,\frac{q^s}{q^{d\l}-1}\right)^{-1} \right)^{y_s}
			= \left( \prod_{s=0}^{d\l-1} \left( \delta_x^{(s)} \right)^{y_s} \right) \ggs (x,y)^{-1}
		\end{align*}
		where $\delta_x^{(s)}$ is defined as in Theorem \ref{thm-first-gkt-formula}:
		Let $0 \leq e \leq \l-1$ be the unique integer such that $ed \leq s < (e+1)d$ and $s' := s - ed$.
		Then $\delta_x^{(s)} := v \anginf{v^{\l-e-1}x}$ if $x_e \in A_{+,s'}$ and $1$ otherwise.
		One sees that the delta part is identical to that in the theorem.
	\end{proof}
	
	\subsubsection{The arithmetic case}
	
	We now recall Thakur's analog of Gross-Koblitz formula for the $v$-adic arithmetic gamma function \cite{thakur1988gauss}.
	Let $\chi := \psi|_{A/v}: A/v \to \Fqd \sbe \ovl{k}$ be the usual Teichmüller character, and choose an $A$-module isomorphism $\phi :A/v \to \Lambda_v$ (recall \S\ref{subsubsection-cyclotomic-function-fields}).
	Then define the \textit{arithmetic Gauss sum} to be
	\begin{equation}    \label{eq-ari-gauss-sum}
		\bags
		:= -\sum_{z\in (A/v)^\times} \chi(z^{-1}) \phi(z) \in K_v \Fqd.
	\end{equation}
	(The same element is denoted as $g_0$ in the original paper.)
	Furthermore, for $y \in (q^{d\l}-1)^{-1} \ZZ$, we write
	$$
	\ang{y} = \sum_{s=0}^{d\l-1} \frac{y_s q^s}{q^{d\l}-1}
	\quad
	(0 \leq y_s < q \text{ for all } s),
	$$
	and define the special monomial of arithmetic Gauss sums as
	$$
	\ags(y) := \prod_{s=0}^{d\l-1} (\bags)^{y_s \tau_q^s}.
	$$
	(Here, $\tau_q$ is the canonical extension of the $q$-th power Frobenius on the constant field.)
	
	Let $\varpi_v \in k_v(\phi(1))$ be the unique $(q^d-1)$-st root of $-v$ such that $\varpi_v \equiv -\phi(1) \pmod{\phi(1)^2}$.
	Then under this setting, Thakur proved the following Gross-Koblitz formula for the $v$-adic arithmetic gamma function (\cite[Theorem VI]{thakur1988gauss} and \cite[Theorem 4.8]{thakur1991gamma}).
	\begin{equation}    \label{eq-gkt-formula-for-ari}
		\ags(y) = \varpi_v^{(q^d-1) \sum_{i=0}^{\l-1} \ang{|v^i|_\infty y}} \prod_{i=0}^{\l-1} \vaf \left(-\ang{|v^i|_\infty y}\right)^{-1}.
	\end{equation}
	In particular, this implies that $\vag(r/(q^d-1))$ is algebraic for all $r \in \ZZ$.
	
	\subsubsection{The two-variable case}
	
	Combining Theorem \ref{thm-gkt-formula-for-geo} with \eqref{eq-gkt-formula-for-ari}, we obtain immediately an analogous formula for the two-variable case.
	Consider the two-variable Gauss sums
	\begin{equation}     \label{eq-two-variable-gauss-sum}
		G_\l (x,y) := \frac{\ggs(x,y)}{\ags(y)}.
	\end{equation}
	
	\begin{thm}[Gross-Koblitz-Thakur formula for two-variable gamma function]     \label{thm-gkt-formula-for-two}
		For any $x \in \nfk^{-1} A$ and $y \in (q^{d\l}-1)^{-1} \ZZ$, write
		$$
		\anginf{x} = \sum_{i=0}^{\l-1} \frac{x_i v^i}{v^\l-1}
		\quad
		(\deg x_i < d \text{ for all } i),
		\quad
		\ang{y} = \sum_{s=0}^{d\l-1} \frac{y_s q^s}{q^{d\l}-1}
		\quad
		(0 \leq y_s < q \text{ for all } s).
		$$
		Then we have
		\begin{equation*}
			G_{\l} (x,y)
			= \left( \prod_{i=0}^{\l-1}
			 \delta_{x,i}^{y_{di+\deg x_i}} \right) 
			\varpi_v^{-(q^d-1) \sum_{i=0}^{\l-1} \ang{|v^i|_\infty y}} \cdot
			\prod_{i=0}^{\l-1} \vtf\left( \anginf{v^ix}, -\ang{|v^i|_\infty y} \right)^{-1}
		\end{equation*}
		where $\delta_{x,i} := v\anginf{v^{\l-i-1}x}$ if $x_i \in A_+$ and $1$ otherwise.
	\end{thm}
	
	\subsubsection{The algebraicity}
	
	It is straightforward to write Theorems \ref{thm-gkt-formula-for-geo} and \ref{thm-gkt-formula-for-two} in terms of special gamma values.
	Thus, as in the classical and arithmetic cases, these formulas yield a class of algebraic special $v$-adic gamma values.
	
	\begin{cor}
		For any $a \in A$ and $r \in \ZZ$, $\vgg(a/(v-1))$, $\vgg(a/(v-1),r/(q^d-1))$ and $\vtg(a/(v-1),r/(q^d-1))$ are algebraic over $k$.
	\end{cor}
	
	\subsection{More on geometric Gauss sums}    \label{subsection-more-on-geometric-gauss-sums}
	
	Now, we use our results in the previous section to establish further arithmetic properties of geometric Gauss sums.
	We keep the same notations as in \S\ref{subsubsection-galois-actions}.
	
	\subsubsection{Reflection formula}
	
	The following theorem is viewed as the reflection formula for geometric Gauss sums.
	
	\begin{thm}[Reflection formula for geometric Gauss sums]    \label{thm-gauss-sum-reflection}
		For any $x \in \nfk^{-1}A \setminus A$, we have
		$$
		\prod_{\epsilon \in \Fqst} \prod_{s=0}^{d\l-1} (\bggs_x)^{\sigma_{\epsilon,s}}
		=\prod_{\epsilon \in \Fqst} \ggs(\epsilon x)
		= v^\l.
		$$
		In particular, the prime factorization of $\bggs_x$ only involves primes above $v$.
	\end{thm}
	
	\begin{eg}
		Consider Example \ref{eg-ggs}.
		When $v = \T, \nfk = v-1$, and $x = 1/\nfk$, Theorem \ref{thm-gauss-sum-reflection} implies
		$$
		\prod_{\epsilon \in \Fqst} \bggs_{\epsilon x}
		= \prod_{\epsilon \in \Fqst} \left( 1-\epsilon(1-\T)^{1/(q-1)} \right)
		= v,
		$$
		which can also be verified directly.
	\end{eg}
	
	\begin{proof}
		The first assertion follows from Proposition \ref{prop-galois-actions}(1) and \eqref{eq-gauss-sum-monomial}.
		For the second, we may assume $\anginf{x} = x$ and write
		$$
		x = \sum_{i=0}^{\l-1} \frac{x_i v^i}{v^\l-1}
		\quad
		(\deg x_i < d \text{ for all } i).
		$$
		By Theorem \ref{thm-gkt-formula-for-geo}, we have
		\begin{equation}     \label{eq-gauss-sum-reflection-1}
			\prod_{\epsilon \in \Fqst} \ggs (\epsilon x)
			= \prod_{\epsilon \in \Fqst} \left( \prod_{i=0}^{\l-1} \delta_{\epsilon x,i} \cdot  \prod_{i=0}^{\l-1} \vgf(\anginf{v^i \epsilon x})^{-1} \right)
		\end{equation}
		where $\delta_{\epsilon x,i} := v\anginf{v^{\l-i-1} \epsilon x}$ if $\epsilon x_i \in A_+$ and $1$ otherwise.
		Note the former case happens if and only if $x_i \neq 0$ and $\epsilon = \sgn( x_i)^{-1}$.
		So the delta part is
		$$
		\prod_{\epsilon \in \Fqst} \prod_{i=0}^{\l-1} \delta_{\epsilon x,i}
		= \prod_{\substack{i=0 \\ x_i \neq 0}}^{\l-1} v \cdot \sgn(x_i)^{-1} \cdot \anginf{v^{\l-i-1}x}.
		$$
		For the gamma (factorial) part, we apply the reflection formula of $\vgf(\cdot)$ \cite[Theorem 4.10.5]{thakur2004function} and see that
		$$
		\prod_{\epsilon \in \Fqst} \prod_{i=0}^{\l-1} \vgf(\anginf{v^i \epsilon x})
		= \prod_{\substack{i=0 \\ x_i \neq 0}}^{\l-1} \sgn(x_i)^{-1} \cdot \anginf{v^{\l-i}x}.
		$$
		Combining these two with \eqref{eq-gauss-sum-reflection-1}, we have
		\begin{equation}   \label{eq-gauss-sum-reflection-2}
			\prod_{\epsilon \in \Fqst} \ggs(\epsilon x)
			= \prod_{\substack{i=0 \\ x_i \neq 0}}^{\l-1} v \cdot \frac{\anginf{v^{\l-i-1}x}}{\anginf{v^{\l-i}x}}.
		\end{equation}
		
		Now, we let $\{x_{i_j} \mid 1 \leq j \leq n\}$ where $0 \leq i_1 < \cdots < i_n \leq \l-1$ be all the non-zero digits in the $v$-adic expansion of the numerator of $x$.
		Then \eqref{eq-gauss-sum-reflection-2} becomes
		\begin{equation}       \label{eq-gauss-sum-reflection-3}
			\prod_{\epsilon \in \Fqst} \ggs(\epsilon x)
			= v^n \prod_{j=1}^n \frac{\anginf{v^{\l-i_j-1}x}}{\anginf{v^{\l-i_j}x}}.
		\end{equation}
		In the last product, one sees that for $1\leq j \leq n-1$, the numerator of the $j$-th term cancels out the denominator of the $(j+1)$-st term, leaving $v^{i_{j+1}-i_j-1}$ in the numerator.
		And similarly, the numerator of the last term cancels out the denominator of the first term, leaving $v^{\l+i_1-i_n-1}$ also in the numerator.
		So \eqref{eq-gauss-sum-reflection-3} becomes
		$$
		\prod_{\epsilon \in \Fqst} \ggs(\epsilon x)
		= v^n \cdot v^{i_2-i_1-1} \cdot v^{i_3-i_2-1} \cdots v^{\l + i_1 - i_n - 1}
		= v^\l.
		$$
		This completes the proof.
	\end{proof}
	
	\begin{cor}
		(1) For any $x \in \nfk^{-1}A$ and $y \in (q^{d\l}-1)^{-1}\ZZ$, we have
		$$
		\ggs(x,y) \ggs(x,1-y) = \ggs(x)^{q-1}.
		$$
		(2) Moreover, if $x \in \nfk^{-1}A \setminus A$, we have
		$$
		\prod_{\epsilon \in \Fqst}
		\ggs(\epsilon x,y) \ggs(\epsilon x,1-y) = v^{\l(q-1)}.
		$$
	\end{cor}
	
	\begin{proof}
		Write $\ang{y} = \sum_{s=0}^{d\l-1} y_sq^s/(q^{d\l}-1)$ where $0\leq y_s < q$ for all $s$.
		Then the first assertion follows from \eqref{eq-gauss-sum-monomial} and the observation that $\ang{1-y} = \sum_{s=0}^{d\l-1} (q-1-y_s)q^s/(q^{d\l}-1)$.
		And the second follows from (1) and Theorem \ref{thm-gauss-sum-reflection}.
	\end{proof}
	
	\subsubsection{\texorpdfstring{$\infty$}{infty}-adic absolute values}
	
	Theorem \ref{thm-gauss-sum-reflection} has another immediate consequence concerning the absolute values of geometric Gauss sums.
	
	\begin{prop}    \label{prop-absolute-values}
		For any $x \in \nfk^{-1}A \setminus A$ and $0 \leq s \leq d\l-1$, the valuation (normalized so that the valuation of $\T$ is $-1$) of $(\bggs_x)^{\tau_q^s}$ at any infinite place of $K_{\nfk,d\l}$ is $-1/(q-1)$.
		In particular, the normalized $\infty$-adic valuation of $\ggs(x)$ is $-d\l/(q-1)$.
	\end{prop}
	
	\begin{proof}
		Choose any infinite place $\td{\infty}$ of $K_{\nfk,d\l}$ above $\infty$.
		Let $K_{\td{\infty}}$ (resp. $k_\infty$) be the completion of $K_{\nfk,d\l}$ at $\td{\infty}$ (resp. $k$ at $\infty$) with normalized valuation $\ord_{\td{\infty}} (\T) = -1$.
		The Galois group $\Gal(K_{\td{\infty}}/k_\infty)$ is isomorphic to $\Fqst \times \ZZ/d\l \ZZ$ (recall \S\ref{subsubsection-cyclotomic-function-fields}).
		So by Theorem \ref{thm-gauss-sum-reflection}, we have
		$$
		\ord_{\td{\infty}} ((\bggs_x)^{\tau_q^s})
		= \frac{1}{[K_{\td{\infty}} : k_\infty]} \ord_{\infty} \left( \Nr_{K_{\td{\infty}}/k_\infty} ((\bggs_x)^{\tau_q^s}) \right)
		=  \frac{1}{(q-1)d\l} \ord_{\infty}(v^\l)
		= -\frac{1}{q-1}.
		$$
	\end{proof}
	
	In \S\ref{subsection-the-behavior-of-geometric-gauss-sums-at-infinity}, we investigate the signs of geometric Gauss sums at infinity.
	
	\subsubsection{Hasse-Davenport relations}     \label{subsubsection-hasse-davenport-relations}
	
	Next, we establish various analogs of Hasse-Davenport product and lifting relations.
	
	\begin{thm}[Hasse-Davenport product relations for geometric Gauss sums]     \label{thm-hd-geo-product}
		Suppose $x \in \nfk^{-1}A$.
		For any $g \in A_{+,h}$ with $(g,v) = 1$, we let $f$ be the order of $v$ modulo $g\nfk$.
		Then for any $y \in (q^{df}-1)^{-1}\ZZ$, we have
		$$
		\prod_{\alpha} \ggsf\left( \frac{x+\alpha}{g},y \right) \bigg/ \ggsf\left( \frac{\alpha}{g},y \right) = \ggsf(x,q^hy)
		$$
		where $\alpha$ runs through a complete residue system modulo $g$.
		In particular, we have
		\begin{equation}    \label{eq-hd-geo-product}
			\prod_{\alpha} \ggsf\left( \frac{x+\alpha}{g} \right) \bigg/ \ggsf\left( \frac{\alpha}{g} \right) = \ggsf(x).
		\end{equation}
	\end{thm}
	
	\begin{proof}
		It suffices to show the case $y = 1/(q^{df}-1)$.
		The general situation will follow by applying the group ring element $\sum_{s=0}^{df-1} y_s\tau_q^s$ to both sides (we put $y = \sum_{s=0}^{df-1} y_sq^s/(q^{df}-1)$ as always).
		Note we may also assume $x \in \nfk^{-1}A \setminus A$ because otherwise the result is trivial.
		Observe that for each $0 \leq i \leq \l-1$, we have
		$$
		\left\{ \lranginf{v^i\left(\frac{x+\alpha}{g}\right)} \Biggm| \deg \alpha < \deg g \right\}
		= \left\{ \frac{\anginf{v^ix}+\alpha}{g} \Biggm| \deg \alpha < \deg g \right\}
		$$
		and
		\begin{equation}     \label{eq-v^i-alpha-g}
			\left\{ \lranginf{v^i \frac{\alpha}{g}} \biggm| \deg \alpha < \deg g \right\}
			= \left\{ \frac{\alpha}{g} \biggm| \deg \alpha < \deg g \right\}.
		\end{equation}
		Using these, Theorems \ref{thm-first-gkt-formula}, \ref{thm-multiplication-formula}, and the translation formula of $\vgf(x,\cdot)$ (see the proof of \cite[Lemma 4.6.2]{thakur2004function}), we have up to explicit rational multiples (for $\kappa_1,\kappa_2 \in \CC_v$, we write $\kappa_1 \sim_k \kappa_2$ if $\kappa_1/\kappa_2 \in k^\times$),
		\begin{alignat*}{2}
			& &&\prod_{\alpha} \ggsf\left( \frac{x+\alpha}{g},y \right) \bigg/ \ggsf\left( \frac{\alpha}{g},y \right)
			\sim_k \prod_{\alpha} \prod_{i=0}^{f-1} \frac{\vgf(\anginf{v^i \alpha/g},-q^{di} y)}{\vgf(\anginf{v^i(x+\alpha)/g},-q^{di} y)}    \\
			&= &&\prod_{i=0}^{f-1} \prod_{\alpha} \frac{\vgf(\alpha/g,-q^{di} y)}{\vgf((\anginf{v^ix}+\alpha)/g,-q^{di} y)}
			\sim_k \prod_{i=0}^{f-1} \vgf\left(\anginf{v^ix},-q^{di} q^hy\right)^{-1}   \\
			&\sim_k &&\prod_{i=0}^{f-1} \vgf\left(\anginf{v^ix}, -\ang{q^{di} q^hy}\right)^{-1}
			\sim_k \ggsf(x,q^hy).
		\end{alignat*}
		Thus, the result is proved up to some $\kappa \in k^\times$.
		
		Since the geometric Gauss sums lie above $v$ by Theorem \ref{thm-gauss-sum-reflection}, we see that $\kappa$ is up to an $\Fqst$-multiple, an integral power of $v$.
		Furthermore, recall we assumed $x \notin A$.
		Also note that $(x+\alpha)/g \notin A$ for all $\alpha$ and $\alpha/g \in A$ if and only if $\alpha=0$.
		So by Proposition \ref{prop-absolute-values}, the $\infty$-adic valuation of $\kappa$ is $0$.
		This shows that $\kappa \in \Fqst$ is a constant.
		Finally, from the explicit expression of $\kappa$, one sees that both of its denominator and numerator are monic polynomials.
		(Note that all the deltas coming from Theorem \ref{thm-first-gkt-formula} and the rational factors coming from the translation formula of $\vgf(x,\cdot)$ have this property.
		Also, recall that $g$ is monic by the hypothesis.)
		Hence, we conclude that $\kappa = 1$.
	\end{proof}
	
	\begin{rem}       \label{rem-hd-v-power}
		Keep the same assumptions as Theorem \ref{thm-hd-geo-product}.
		Note that by Theorem \ref{thm-gauss-sum-reflection}, the second identity \eqref{eq-hd-geo-product} can be rewritten as
		$$
		\prod_{\alpha} \ggsf\left( \frac{x+\alpha}{g} \right) 
		= v^{f(q^h-1)/(q-1)} \ggsf(x).
		$$
	\end{rem}
	
	The following result is a two-variable analog of Theorem \ref{thm-hd-geo-product}.
	Note that this gives an intriguing correlation between geometric and arithmetic Gauss sums.
	
	\begin{thm}[Hasse-Davenport product relation for two-variable Gauss sums]    \label{thm-hd-two-product}
		Suppose $x \in \nfk^{-1}A$.
		For any $g \in A_{+,h}$ with $(g,v) = 1$, we let $f$ be the order of $v$ modulo $g\nfk$.
		Then for any $y \in (q^{df}-1)^{-1}\ZZ$, we have
		$$
		\prod_{\alpha} G_f \left(\frac{x+\alpha}{g}, y\right)
		= g_v^{-q^h(q^{df}-1)\ang{y}} \cdot G_f(x,q^hy),
		$$
		where $g_v \in \Fqd^\times \sbe A_v^\times$ is the Teichmüller representative of $g$ in $A_v$.
		Equivalently,
		$$
		\prod_{\alpha} G_f^{\textnormal{geo}} \left(\frac{x+\alpha}{g}, y\right) \bigg/ G_f^{\textnormal{ari}}(y)
		= g_v^{-q^h(q^{df}-1)\ang{y}} \cdot G_f^{\textnormal{geo}} (x,q^hy) \Big/ G_f^{\textnormal{ari}} (q^hy).
		$$
	\end{thm}
	
	\begin{proof}
		The equivalence of the two identities follows from \eqref{eq-two-variable-gauss-sum}.
		Our strategy is similar to the proof of Theorem \ref{thm-hd-geo-product}.
		Note that we may again assume $y = 1/(q^{df}-1)$.
		Applying \eqref{eq-two-variable-gauss-sum}, Theorems \ref{thm-hd-geo-product}, \ref{thm-first-gkt-formula}, \eqref{eq-gkt-formula-for-ari}, \eqref{eq-v^i-alpha-g}, the translation formula of $\vaf(\cdot)$ (see the proof of \cite[Lemma 4.6.2]{thakur2004function}), and Theorem \ref{thm-multiplication-formula} sequentially, we have
		\begin{alignat*}{2}
			& &&\prod_{\alpha} G_f \left(\frac{x+\alpha}{g}, y\right) \bigg/ G_f(x,q^hy)
			= \left( \prod_{\alpha} G_f^{\textnormal{geo}} \left(\frac{\alpha}{g}, y \right) \right)
			\left( \frac{G_f^{\textnormal{ari}}(q^hy)}{G_f^{\textnormal{ari}}(y)^{q^h}} \right)   \\
			&\sim_k &&\left( \prod_{\alpha} \prod_{i=0}^{f-1} \vgf\left( \anginf{v^i \alpha/g}, -\ang{q^{di} y} \right)^{-1} \right)
			\left( \prod_{i=0}^{f-1} \frac{\vaf(-\ang{q^{di} y})^{q^h}}{\vaf(-\ang{q^{di} q^hy})}\right)  \varpi_v^{q^{h'}-q^h}    \\
			&\sim_k &&\left( \frac{\vaf(1/(1-q^d))^{q^h}}{\prod_\alpha \vgf( \alpha/g , 1/(1-q^d))} \cdot \frac{1}{\vaf(q^h/(1-q^d))} \right) \varpi_v^{q^{h'}-q^h} \\
			&\sim_k &&\left(g^{q^d-1}\right)^{-\partial_v(q^h/(1-q^d))} \varpi_v^{q^{h'}-q^h}.
		\end{alignat*}
		Here, $h'$ is the unique integer such that $0 \leq h' \leq d-1$ and $h' \equiv h \pmod{d}$.
		Note that
		$$
		\left(g^{q^d-1}\right)^{-\partial_v(q^h/(1-q^d))}
		=
		\begin{cases}
			g^{q^{h-d}} g_v^{-q^h}, & h \geq d, \\
			g^{q^h} g_v^{-q^h}, & h < d.
		\end{cases}
		$$
		On the other hand, write $h = h' + cd$ for some $c \in \NN \cup \{0\}$.
		Then $q^{h'} - q^h = -(q^d-1)e$ where $e := q^{h'} (1+q^d + \cdots + q^{(c-1)d})$ and so
		\begin{equation}    \label{eq-hd-two-product-varpi}
			\varpi_v^{q^{h'}-q^h}
			= \left(\varpi_v^{q^d-1}\right)^{-e}
			= (-v)^{-e} \in k.
		\end{equation}
		Thus, the result is proved up to some $\kappa \in k^\times$.
		
		Now, since both Gauss sums lie above $v$ (Theorem \ref{thm-gauss-sum-reflection} and \cite[Theorem II]{thakur1988gauss}), we see that $\kappa$ is up to an $\Fqst$-multiple, an integral power of $v$.
		Furthermore, since the absolute values of geometric and arithmetic Gauss sums are identical (Proposition \ref{prop-absolute-values} and \cite[Theorem IV]{thakur1988gauss}), we see that the $\infty$-adic valuation of $\kappa$ is $0$.
		This implies that $\kappa \in \Fqst$ is a constant.
		Finally, from the explicit expression of $\kappa$, one sees that
		$$
		\kappa = (-1)^c (-1)^e = 1,
		$$
		where $(-1)^c$ comes from applying the translation formula of $\vaf(\cdot)$ and $(-1)^e$ is the sign of \eqref{eq-hd-two-product-varpi}.
		(We once again note that all the deltas coming from Theorem \ref{thm-first-gkt-formula} and the integral power of $g$ coming from Theorem \ref{thm-multiplication-formula} have sign $1$.)
	\end{proof}
	
	We also have the following product relation with respect to $y$.
	
	\begin{thm}
		Suppose $y \in N^{-1}\ZZ$ for some $N \in \NN$ with $(N,q) = 1$.
		For any $n\in\NN$ with $(n,q) = 1$, we let $f$ be the order of $q^d$ modulo $Nn$.
		Then for any $x \in (v^f-1)^{-1}A$, we have
		$$
		\prod_{i=0}^{n-1} \ggsf\left(x,\frac{y+i}{n}\right)
		= \ggsf(x)^{(n-1)(q-1)/2} \ggsf(x,y).
		$$
	\end{thm}
	
	\begin{proof}
		We define a function $g$ on $\ZZ_p$ by
		$$
		g\left(1 + \sum_{i=0}^\infty y_iq^i\right) := \prod_{i=0}^{df-1} \left( (\bggs_x)^{\tau_q^i} \right)^{y_i}
		\text{ where }
		0 \leq y_i < q
		\text{ for all }
		i.
		$$
		Then $g$ fits into the framework of \cite[Section 2]{thakur1991gamma}, and we have for all $y \in (q^{df}-1)^{-1}\ZZ$,
		$$
		\ggsf(x,y) = g(1 - \ang{y}).
		$$
		Now, observe that
		$$
		\left\{ 1 - \lrang{\frac{y+i}{n}} \biggm| 0\leq i \leq n-1 \right\}
		= \left\{ \frac{(1-\ang{y})+i}{n} \biggm| 0\leq i \leq n-1 \right\}.
		$$
		So by \cite[Lemma 2.4]{thakur1991gamma},
		\begin{align*}
			\prod_{i=0}^{n-1} \ggsf\left(x,\frac{y+i}{n}\right)
			&= \prod_{i=0}^{n-1} g\left(1-\lrang{\frac{y+i}{n}}\right)
			= \prod_{i=0}^{n-1} g\left(\frac{(1-\ang{y})+i}{n}\right)   \\
			&= g(0)^{(n-1)/2} g(1-\ang{y})
			= \ggsf(x)^{(n-1)(q-1)/2} \ggsf(x,y).
		\end{align*}
	\end{proof}
	
	We also establish an analog of Hasse-Davenport lifting relation.
	Compare it with Proposition \ref{prop-compatibility-of-ggs} and \S\ref{subsubsection-an-additive-hasse-davenport-lifting-relation}.
	See the corresponding result of the arithmetic case \cite[Theorem VIII]{thakur1988gauss} also.
	
	\begin{thm}[Hasse-Davenport lifting relation for geometric Gauss sums]
		Suppose $x \in \nfk^{-1}A \sbe (\nfk')^{-1}A$ where $\nfk \mid \nfk'$ with $(\nfk',v) = 1$.
		Let $\l'$ be the order of $v$ modulo $\nfk'$ and put $m := \l' / \l \in \NN$.
		Then for any $y \in (q^{d\l}-1)^{-1}\ZZ$, we have
		$$
		G_{\l'}^{\textnormal{geo}}(x,y) = \ggs(x,y)^m.
		$$
	\end{thm}
	
	\begin{proof}
		This follows immediately from \eqref{eq-gauss-sum-monomial} because $\bggs_x$ is fixed by $\tau_q^{d\l}$ and
		$$
		\sum_{s=0}^{d\l-1} \frac{y_sq^s}{q^{d\l}-1}
		= \sum_{j=0}^{m-1} q^{jd\l} \cdot \frac{y_0 + y_1 q+ \cdots + y_{d\l-1} q^{d\l-1}}{q^{d\l'}-1}.
		$$
	\end{proof}
	
	\subsubsection{Stickelberger's theorem}
	
	Finally, we consider the prime factorizations of geometric Gauss sums, which is an analog of the classical Stickelberger's theorem.
	Let $\Ocal_{\nfk}$ and $\Ocal_{\nfk,d\l}$ be the ring of integers of $K_\nfk$ and $K_{\nfk,d\l}$, respectively.
	We let $\Pfk_{\nfk,d\l}$ be the prime in $K_{\nfk,d\l}$ above $v$ corresponding to the inclusions $K_{\nfk,d\l} \sbe \ovl{k} \sbe \CC_v$, and $\Pfk_\nfk := \Pfk$ be the prime in $K_\nfk$ below $\Pfk_{\nfk,d\l}$ (recall \S\ref{subsubsection-arithmetic-and-geometric-characters}).
	
	\begin{thm}[Stickelberger's theorem for geometric Gauss sums]     \label{thm-stickelberger}
		Given $x \in k$ with $0 < |x|_\infty < 1$, write $x = a_0/\nfk$ where $\deg a_0 < \deg \nfk$ and $(a_0,\nfk)=1$.
		Define
		$$
		\eta_{x,\nfk,d\l} := \sigma_{a_0,\deg \nfk} \cdot \eta_{\nfk,d\l}
		\quad
		\text{where}
		\quad
		\eta_{\nfk,d\l} := \sum_{\substack{a\in A_+ \\ \deg a < \deg \nfk \\ (a,\nfk)=1}}  \sigma_{a,\deg a}^{-1} \in \ZZ[\Gal(K_{\nfk,d\l}/k)].
		$$
		Then $\bggs_x$ has prime factorization
		$$
		\bggs_x \cdot \Ocal_{\nfk,d\l}
		= \Pfk_{\nfk,d\l}^{\eta_{x,\nfk,d\l}}.
		$$
		Consequently, let
		$$
		\eta_{x,\nfk} := \rho_{a_0} \cdot \eta_{\nfk}
		\quad
		\text{where}
		\quad
		\eta_{\nfk} := \sum_{\substack{a\in A_+ \\ \deg a < \deg \nfk \\ (a,\nfk)=1}}  \rho_a^{-1} \in \ZZ[\Gal(K_{\nfk}/k)].
		$$
		Then $\ggs(x)$ has prime factorization
		$$
		\ggs(x) \cdot \Ocal_{\nfk}
		= \Pfk_{\nfk}^{\eta_{x,\nfk}}.
		$$
	\end{thm}
	
	\begin{proof}
		It suffices to show the case $a_0 = 1$.
		Put $\mfk := v^\l-1$ and write
		$$
		\frac{1}{\nfk} = \frac{b_0}{\mfk}.
		$$
		Fix any $a \in A_+$, $\deg a< \deg \nfk$ with $(a,\nfk) = 1$.
		From Theorem \ref{thm-first-gkt-formula}, we take
		$$
		x
		= \frac{a}{\nfk}
		= \frac{ab_0}{\mfk}
		=: \sum_{i=0}^{e} \frac{x_i v^i}{v^\l-1}
		\quad
		\text{and}
		\quad
		y = \frac{q^{\deg ab_0}}{q^{d\l}-1}
		$$
		where $0 \leq e \leq \l-1$, $x_e \in A_+$, and $\deg x_i < d$ for all $i$.
		Then we have
		$$
		\ggs (x,y)
		= \left(\bggs_{ab_0/\mfk}\right)^{\sigma_{1,\deg ab_0}}
		= v\anginf{v^{\l-e-1}x} \cdot
		\prod_{i=0}^{\l-1} \vgf \left( \anginf{v^ix}, -\ang{|v^i|_\infty y} \right)^{-1}.
		$$
		By Proposition \ref{prop-galois-actions}(1), this implies
		$$
		\left(\bggs_{1/\nfk}\right)^{\sigma_{a,\deg ab_0}}
		= v^{\l-e} \frac{ab_0}{\mfk} \cdot \prod_{i=0}^{\l-1} \vgf \left( \anginf{v^ix}, -\ang{|v^i|_\infty y} \right)^{-1}.
		$$
		From here we take the $\Pfk_{\nfk,d\l}$-adic valuations to both sides, using the fact that $\vgf(\cdot , \cdot)$ is a unit in $A_v$, we get
		$$
		\ord_{\Pfk_{\nfk,d\l}} \left( \left(\bggs_{1/\nfk}\right)^{\sigma_{a,\deg ab_0}} \right)
		= \ord_{\Pfk_{\nfk,d\l}} \left( v^{\l-e} \frac{ab_0}{\mfk} \right)
		= \l - e + \ord_v(ab_0).
		$$
		
		We will claim that this quantity is exactly the number of elements of the form $\sigma_{b,\deg b}$ in the coset $\sigma_{a,\deg a}D$ (the decomposition group of $v$ in $K_{\nfk,d\l}$) of $\Gal(K_{\nfk,d\l}/k)$, where $b \in A_+$, $\deg b< \deg \nfk$ and $(b,\nfk) = 1$.
		Assuming this for a moment, then since every such $\sigma_{b,\deg b}$ results in the same prime as $\sigma_{a,\deg a}$ does after applying to $\Pfk_{\nfk,d\l}$, and since this holds for each $a \in A_+$, $\deg a< \deg \nfk$ with $(a,\nfk) = 1$, we obtain
		\begin{equation}      \label{eq-divisors-of-geometric-gauss-sum}
			\Pfk_{\nfk,d\l}^{\sigma_{1,\deg b_0}^{-1} \eta_{\nfk,d\l}}
			\Bigm| 
			\bggs_{1/\nfk} \cdot\Ocal_{\nfk,d\l}.
		\end{equation}
		Now, applying both sides by $\sum_{\epsilon \in\Fqst} \sum_{s=0}^{d\l-1} \sigma_{\epsilon,s}$ and using Theorem \ref{thm-gauss-sum-reflection}, we have
		$$
		\left( \Pfk_{\nfk,d\l}^{\sigma_{1,\deg b_0}^{-1} \eta_{\nfk,d\l}} \right)^{\sum_{\epsilon \in\Fqst} \sum_{s=0}^{d\l-1} \sigma_{\epsilon,s}} 
		\Biggm|
		\left(\bggs_{1/\nfk}\right)^{\sum_{\epsilon \in\Fqst} \sum_{s=0}^{d\l-1} \sigma_{\epsilon,s}} \cdot\Ocal_{\nfk,d\l}
		= v^\l \cdot\Ocal_{\nfk,d\l}.
		$$
		Since both sides consist of the same amount of primes (counting multiplicity) in $K_{\nfk,d\l}$ (recall Remark \ref{rem-fixed-field} that $v$ splits into $[K_{\nfk}:k]d$ primes in $K_{\nfk,d\l}$), we see in fact they are identical.
		This means \eqref{eq-divisors-of-geometric-gauss-sum} actually gives the equality
		$$
		\bggs_{1/\nfk} \cdot\Ocal_{\nfk,d\l}
		= \Pfk_{\nfk,d\l}^{\sigma_{1,\deg b_0}^{-1} \eta_{\nfk,d\l}},
		$$
		as any other prime divisors will violate the factorization of $v^\l$.
		And this is what we want.
		(Note $\deg b_0 = \deg \mfk - \deg \nfk = d\l - \deg\nfk \equiv -\deg \nfk \pmod{d\l}$.)
		
		It remains to prove the claim.
		In fact, we prove a stronger result that for any $a \in A_+$ with $\deg a < \deg \mfk$, if
		$$
		a = a_rv^r + \cdots + a_ev^e
		\quad
		(0 \leq r \leq e \leq \l-1)
		$$
		is written in $v$-adic expansion, where $r = \ord_v(a)$, $\deg a_i < d$ for all $i$, and $a_e \in A_+$, then the number of $b \in A_+$, $\deg b< \deg \mfk$ satisfying
		\begin{equation}     \label{eq-condition-on-b}
			b \equiv av^i \Mod{\mfk}
			\text{ and }
			\deg b \equiv \deg a + di \Mod{d\l}
			\text{ for some }
			0 \leq i \leq \l-1
		\end{equation}
		is precisely $\l - e + r$.
		Indeed, one sees that \eqref{eq-condition-on-b} is satisfied if and only if $0 \leq i \leq \l-e-1$ or $\l-r \leq i \leq \l-1$.
		This gives $\l-e+r$ elements in total.		
		
		Finally, to apply the claim to our situation, observe that there is a bijection from
		$$
		\{b \mid b\in A_+, \deg b < \deg \nfk, b \equiv av^i \Mod{\nfk}, \deg b \equiv \deg a + di \Mod{d\l}\}
		$$
		to
		$$
		\{b \mid b\in A_+, \deg b < \deg \mfk, b \equiv ab_0v^i \Mod{\mfk}, \deg b \equiv \deg ab_0 + di \Mod{d\l}\}
		$$
		sending $b$ to $bb_0$.
	\end{proof}
	
	\begin{rem}
		The connection between special $v$-adic gamma values and the Stickelberger elements in Theorem \ref{thm-stickelberger} has already been pointed out by Thakur in his Gross-Koblitz theorem for the geometric case \cite[Section 8.6]{thakur2004function}.
		Our Theorem \ref{thm-stickelberger} provides a more explicit description in terms of geometric Gauss sums and generalizes it to the two-variable case.
	\end{rem}
	
	\section{Pairing comparisons and geometric Gauss sums at infinity}
	
	\subsection{The functions \texorpdfstring{$e$}{e} and \texorpdfstring{$e^*$}{e*}}
	
	Recall that in Theorem \ref{thm-restatement-of-abp-5.4.4}, we gave an algebraic reformulation of \cite[Theorem 5.4.4]{abp2004determination}.
	The original statement, along with its proof, made use of two special functions $e$ and $e^*$, whose definitions and properties are reviewed below.
	We let $k_\infty := \Fq(\!(1/\T)\!)$ be the completion of $k$ at the infinite place and $\CC_\infty$ be the completion of a fixed algebraic closure $\ovl{k}_\infty$ of $k_\infty$.
	We choose an identification from $\ovl{k}$ (fixed in \S\ref{section-preliminaries}) into $\ovl{k}_\infty$.
	
	\subsubsection{The exponential function of the Carlitz module}
	
	Recall \S\ref{subsubsection-carlitz-modules} that the Carlitz module over $\ovl{k}$ is the $\Fq$-algebra homomorphism $C: A \to \ovl{k}\{\tau\}$ given by $C_\T := \T + \tau$.
	The associated power series $\exp_C(z) \in \CC_\infty[\![z]\!]$, called the \textit{Carlitz exponential function}, is $\Fq$-linear and entire on $\CC_\infty$.
	Its kernel is a free rank-one $A$-module generated by the \textit{Carlitz period}
	\begin{equation}         \label{eq-period}
		\td{\pi} := -\td{\T}^q \prod_{i=1}^{\infty} \left( 1-\frac{\T}{\T^{q^i}} \right)^{-1} \in k_\infty(\td{\T})
	\end{equation}
	where $\td{\T} = (-\T)^{1/(q-1)}$ is a fixed $(q-1)$-st root of $-\T$.
	
	We put $e(z) := \exp_C(\td{\pi}z)$.
	Then we have the functional equation
	\begin{equation}    \label{eq-functional-equation-of-e}
		C_a(e(z)) = e(az)
		\quad
		\text{for all}
		\quad
		a \in A.
	\end{equation}
	For any $\nfk \in A_+$, from \eqref{eq-functional-equation-of-e} and the fact that $\ker e= A$, we see that the collection
	$$
	\Lambda_\nfk = \{ e(a/\nfk) \mid a \in A, \deg a < \deg \nfk \}
	$$
	is the $\nfk$-torsion points of the Carlitz module $C(\ovl{k})$, which is isomorphic to $A/\nfk$. 
	Thus, the element $e(a/\nfk)$ is an $A$-generator of $\Lambda_{\nfk}$ if and only if $(a,\nfk)=1$.
	
	\subsubsection{The exponential function of the adjoint Carlitz module}
	
	Next, recall \S\ref{subsubsection-adjoint-carlitz-modules} that the adjoint Carlitz module over $\ovl{k}$ is the $\Fq$-algebra homomorphism $C^*: A \to \ovl{k}\{\tau^{-1}\}$ given by $C^*_\T := \T + \tau^{-1}$.
	There is an analog of the exponential function for $C^*$ with properties parallel to $e(z)$.
	Let $\Res: k_\infty \to \Fq$ be the usual residue map for the variable $\T$.
	In other words, it is the unique $\Fq$-linear functional with kernel $\Fq[\T] + (1/\T^2)\Fq[\![1/\T]\!]$ and $\Res(1/\T) = 1$.
	Let $t$ be an independent variable.
	Set
	$$
	\Omega(t) := \td{\T}^{-q} \prod_{i=1}^\infty \left( 1-\frac{t}{\T^{q^i}} \right)
	\in k_\infty(\td{\T}) [\![t]\!]
	$$
	and
	$$
	\Omega^{(-1)}(t) := \td{\T}^{-1} \prod_{i=0}^\infty \left( 1-\frac{t}{\T^{q^i}} \right) =: \sum_{i=0}^\infty c_it^i
	\in k_\infty(\td{\T}) [\![t]\!].
	$$
	One checks that $\Omega$ satisfies the functional equation $\Omega^{(-1)} = (t-\T) \Omega$.
	
	For $z \in k_\infty$, we set
	$$
	e^*(z) := \sum_{i=0}^\infty \Res(\T^iz)c_i.
	$$
	Then it is a fact that $e^*$ is $\Fq$-linear with $\ker e^* = A$.
	Moreover, from the functional equation $\Omega^{(-1)} = (t-\T) \Omega$, one deduces that (cf. \eqref{eq-functional-equation-of-e})
	$$
	C_a^*(e^*(z)^q) = e^*(az)^q
	\quad
	\text{for all}
	\quad
	a \in A.
	$$
	In particular, for any $\nfk \in A_+$, $e^*(a/\nfk)$ are $q$-th roots of the $\nfk$-torsion points of the adjoint Carlitz module $C^*(\ovl{k})$ for all $a \in A$.
	In other words,
	$$
	\Lambda_\nfk^* = \{ e^*(a/\nfk)^q \mid a \in A, \deg a < \deg \nfk \} \sbe K_\nfk,
	$$
	where the containment is due to Theorem \ref{thm-goss-1.7.11}. 
	
	\subsubsection{A duality identity between torsions of \texorpdfstring{$e$}{e} and \texorpdfstring{$e^*$}{e*}}    \label{subsubsection-a-duality-identity-between-torsions-ofe-and-e^*}
	
	For each integer $N\geq 0$, recall \eqref{eq-Psi} that $\Psi_N(z)$ is defined as
	$$
	1 + \Psi_N(z)
	= \prod_{a\in A_{+,N}} \left(1+\frac{z}{a}\right).
	$$
	We can now state \cite[Theorem 5.4.4]{abp2004determination}.
	
	\begin{thm}    \label{thm-original-abp-5.4.4}
		Fix $\nfk \in A_+$ of positive degree and $\nfk$-dual families 
		$$
		\{a_i\}_{i=1}^{\deg\nfk},
		\quad
		\{b_j\}_{j=1}^{\deg\nfk}.
		$$
		Then for $a_0 \in A$ with $\deg a_0 < \deg \nfk$, we have    \\
		(1)
		$$
		\sum_{i=1}^{\deg \nfk} e^*(a_i/\nfk)^{q^{N+1}} e(b_ia_0/\nfk) = -\Psi_N(a_0/\nfk)
		$$
		for all integers $N\geq 0$.    \\
		(2) Moreover, if $a_0 \in A_+$, we have
		$$
		\sum_{i=1}^{\deg \nfk} e^*(a_i/\nfk) e(b_ia_0/\nfk)^{q^{\deg \nfk - \deg a_0 -1}} = 1.
		$$
	\end{thm}
	
	Let us now briefly recall the terminologies in \S\ref{subsection-a-duality-identity} using the notions of this section.
	For any $\nfk \in A_+$, define the residue pairing with respect to $\nfk$ (which is perfect) as
	$$
	\angres{\cdot,\cdot}: A/\nfk \times A/\nfk \to \Fq,
	\quad
	\angres{a,b} := \Res(ab/\nfk).
	$$
	Let $\{a_i\}_{i=1}^{\deg\nfk}, \{b_j\}_{j=1}^{\deg\nfk}$ be any fixed $\nfk$-dual families so that we have $\Res(a_ib_j/\nfk) = \delta_{ij}$.
	Using the function $e(z)$, we choose an $A$-module isomorphism from $A/\nfk$ to $\Lambda_\nfk$, which sends $1 \in A/\nfk$ to $\lambda = e(\alpha/\nfk) \in \Lambda_\nfk \sbe \ovl{k}_\infty$ for some $\alpha\in A$ with $(\alpha,\nfk)=1$.
	By \eqref{eq-functional-equation-of-e}, we put
	\begin{equation}    \label{eq-lambda-j}
		\lambda_j := C_{b_j}(\lambda) = e(\alpha b_j/\nfk).
	\end{equation}
	Note that $\{\lambda_j\}_{j=1}^{\deg\nfk}$ forms an $\Fq$-basis of $\Lambda_\nfk$.
	So by Ore's formula, we obtain an $\Fq$-basis $\{\lambda^*_i\}_{i=1}^{\deg\nfk}$ of $\Lambda_\nfk^*$ given by (recall $\Delta(x_1,\ldots,x_n) := \det_{1 \leq i,j \leq n} x_j^{q^{i-1}}$ is the Moore determinant)
	$$
	\lambda_i^* = (-1)^{\deg\nfk+i} \left(\frac{\Delta_i}{\Delta}\right)^q,
	$$
	where
	$$
	\Delta_i := \Delta(\lambda_1,\ldots,\lambda_{i-1},\lambda_{i+1},\ldots,\lambda_{\deg\nfk})
	\quad
	\text{and}
	\quad
	\Delta := \Delta(\lambda_1,\ldots,\lambda_{\deg\nfk}) \neq 0.
	$$
	Then under this setting, we have the following restatement of Theorem \ref{thm-original-abp-5.4.4}.
	
	\begin{thm}      \label{thm-restatement-of-abp-5.4.4-last-section}
		Fix $\nfk \in A_+$ of positive degree and $\nfk$-dual families 
		$$
		\{a_i\}_{i=1}^{\deg\nfk},
		\quad
		\{b_j\}_{j=1}^{\deg\nfk}.
		$$
		Then for $a_0 \in A$ with $\deg a_0 < \deg \nfk$, we have   \\
		(1)
		$$
		\sum_{i=1}^{\deg \nfk} (\lambda_i^*)^{q^N} C_{a_0}(\lambda_i) = -\Psi_N(a_0/\nfk)
		$$
		for all integers $N\geq 0$.   \\
		(2) Moreover, if $a_0 \in A_+$, we have
		$$
		\sum_{i=1}^{\deg \nfk} \lambda_i^* C_{a_0}(\lambda_i)^{q^{\deg \nfk - \deg a_0}} = 1.
		$$
	\end{thm}
	
	Note that by \eqref{eq-lambda-j} and \eqref{eq-functional-equation-of-e}, we have
	$$
	C_{a_0}(\lambda_j)
	= C_{a_0}(e(\alpha b_j/\nfk))
	= e(a_0 \alpha b_j/\nfk).
	$$
	On the other hand, note that if
	$$
	\{a_i\}_{i=1}^{\deg\nfk},
	\quad
	\{b_j\}_{j=1}^{\deg\nfk}
	$$
	is a pair of $\nfk$-dual families, then so is the pair
	$$
	\{\alpha' a_i\}_{i=1}^{\deg\nfk},
	\quad
	\{\alpha b_j\}_{j=1}^{\deg\nfk}
	$$
	where $\alpha' \in A$ with $\alpha\alpha' \equiv 1 \pmod{\nfk}$.
	Thus, to justify the equivalence between Theorems \ref{thm-original-abp-5.4.4} and \ref{thm-restatement-of-abp-5.4.4-last-section}, we need to show that
	\begin{equation}      \label{eq-the-equation-of-the-main-goal}
		\lambda_i^*
		= (-1)^{\deg\nfk+i} \left(\frac{\Delta_i}{\Delta}\right)^q
		= e^*(\alpha' a_i/\nfk)^q.
	\end{equation}
	This will be done in the next section by investigating the compatibility of three different pairings: residue pairing, Poonen pairing, and trace pairing.
	
	\subsection{Pairing comparisons}      \label{subsection-pairing-comparisons}
	
	\subsubsection{Residue pairing and trace pairing}        \label{subsubsection-residue-pairing-and-trace-pairing}
	
	For now, we consider the special case $\nfk = v - 1$ where $v \in A_{+,d}$ is irreducible.
	Recall that by Proposition \ref{prop-reduction-of-lambda}, every element in $\FF_\Pfk \simeq \Fqd$ is represented by a unique $\lambda \in \Lambda_\nfk$, and also by a unique $\lambda^* \in \Lambda_\nfk^*$ by Proposition \ref{prop-reduction-of-lambda*}.
	We consider the trace pairing
	$$
	\angtr{\cdot,\cdot}: \FF_\Pfk \times \FF_\Pfk \to \Fq,
	\quad
	\angtr{a,b} := \Tr_{\FF_\Pfk/\Fq}(ab),
	$$
	where we lift the image uniquely to the constant field $\Fq$.
	
	\begin{prop}      \label{prop-lambdai*-and-trace-pairing}
		Suppose $\nfk = v - 1 \in A_{+,d}$.
		Let $\{\lambda_j\}_{j=1}^{d}$ be any $\Fq$-basis of $\Lambda_\nfk$, and set
		$$
		\lambda_i^* 
		:= (-1)^{d+i} \left(\frac{\Delta_i}{\Delta}\right)^q.
		$$
		Then one has
		$$
		\angtr{\ovl{\lambda}{}^*_i , \ovl{\lambda}_j} = \delta_{ij}.
		$$
		In other words, after reduction the $\lambda_i^*$ given by Ore's formula are the dual bases of $\lambda_j$ with respect to trace pairing.
	\end{prop}
	
	\begin{proof}
		We start with noticing that
		$$
		\Delta^q 
		= \left(\det_{1 \leq i,j \leq d} \lambda_j^{q^{i-1}}\right)^q
		= \det_{1 \leq i,j \leq d} \lambda_j^{q^i}
		\equiv (-1)^{d-1} \Delta \pmod{\Pfk}
		$$
		where the power of $-1$ comes from moving the last row to the first.
		On the one hand, this implies (note $\Delta \not\equiv 0 \pmod{\Pfk}$ as $\{\ovl{\lambda}_1,\ldots,\ovl{\lambda}_d\}$ is still an $\Fq$-basis of $\FF_{\Pfk}$)
		$$
		\lambda_i^* 
		= (-1)^{d+i} \left(\frac{\Delta_i}{\Delta}\right)^q
		\equiv (-1)^{i-1} \frac{\Delta_i^q}{\Delta}   \pmod{\Pfk},
		$$
		and on the other, we have by induction that for any $k\in\NN$,
		$$
		\Delta^{q^k} \equiv (-1)^{k(d-1)} \Delta \pmod{\Pfk}.
		$$
		These two imply
		\begin{equation}       \label{eq-trace-of-lambda*-and-lambda}
			\sum_{k=0}^{d-1} (\lambda_i^* \lambda_j)^{q^k}
			\equiv \sum_{k=0}^{d-1} \left( (-1)^{i-1} \frac{\Delta_i^q}{\Delta} \lambda_j \right)^{q^k}
			\equiv \frac{1}{\Delta} \sum_{k=0}^{d-1} (-1)^{(i-1)+k(d-1)} (\Delta_i^q \lambda_j)^{q^k}
			\pmod{\Pfk}.
		\end{equation}
		
		Consider the matrices
		$$
		\Mcal \equiv \left( \lambda_j^{q^{i-1}}  \right),
		\quad  
		\Zcal \equiv (z_{ij})
		\pmod{\Pfk}
		$$
		and the system of linear equations
		\begin{equation}    \label{eq-linear-equations-MZ-I}
			\Mcal \cdot \Zcal \equiv I_d   \pmod{\Pfk}.
		\end{equation}
		By Cramer's rule, one deduces that
		\begin{equation}      \label{eq-zij}
			z_{ij}
			\equiv (-1)^{(i+j)+(j-1)(d-j)} \frac{\Delta_i^{q^j}}{\Delta}
			\pmod{\Pfk}.
		\end{equation}
		Exchanging the product of matrices in \eqref{eq-linear-equations-MZ-I}, it follows that
		$$
		\delta_{ij} 
		\equiv \sum_{k=1}^d z_{ik} \lambda_j^{q^{k-1}}
		\overset{\eqref{eq-zij}}{\equiv} \sum_{k=1}^d (-1)^{(i+k)+(k-1)(d-k)} \frac{\Delta_i^{q^k}}{\Delta} \lambda_j^{q^{k-1}}
		\overset{\eqref{eq-trace-of-lambda*-and-lambda}}{\equiv} \sum_{k=0}^{d-1} (\lambda_i^* \lambda_j)^{q^k}  \pmod{\Pfk}.
		$$
		And this is what we want.
	\end{proof}
	
	The relation between $\nfk$-dual families $\{a_i\}_{i=1}^{d}, \{b_j\}_{j=1}^{d}$ and the elements
	$$
	\{\ovl{e^*(\alpha' a_i/\nfk)^q}\}_{i=1}^{d}, 
	\quad
	\{\ovl{e(\alpha b_j/\nfk)}\}_{j=1}^{d}
	$$
	in the finite field $\FF_\Pfk$ is also explained by the trace pairing.
	Precisely, we have
	
	\begin{prop}    \label{prop-residue-and-trace}
		Suppose $\nfk = v - 1 \in A_{+,d}$.
		Let $\{a_i\}_{i=1}^{d}, \{b_j\}_{j=1}^{d}$ be $\nfk$-dual families.
		Take $\alpha \in A$ prime to $\nfk$ and let $\alpha' \in A$ be such that $\alpha\alpha' \equiv 1 \pmod{\nfk}$.
		Then one has
		$$
		\angtr{ \ovl{e^*(\alpha' a_i/\nfk)^q}, \ovl{e(\alpha b_j/\nfk)} }= \delta_{ij}.
		$$
		In other words, the functions $e^*$ and $e$ together with the reduction modulo $\Pfk$ map dual bases of the residue pairing to dual bases of the trace pairing.
	\end{prop}
	
	Note that Propositions \ref{prop-lambdai*-and-trace-pairing} and \ref{prop-residue-and-trace} imply the main goal \eqref{eq-the-equation-of-the-main-goal} for the special case $\nfk = v-1$.
	
	\begin{proof}
		As $\{\alpha' a_i\}_{i=1}^d, \{\alpha b_j\}_{j=1}^d$ is also a pair of $\nfk$-dual families, we may assume without loss of generality that $\alpha = 1$.
		Let $\lambda_i^* := e^*(a_i/\nfk)^q$ and $\lambda_j := e(b_j/\nfk)$.
		Consider the matrix multiplication
		\begin{equation}    \label{eq-product-in-residue-and-trace}
			\left( \lambda_j^{q^{i-1}} \right)
			\left( (\lambda_i^*)^{q^{j-1}} \right)
			\pmod{\Pfk}.
		\end{equation}
		From Theorem \ref{thm-original-abp-5.4.4}(1), we have for all integers $N \geq 0$,
		$$
		\sum_{k=1}^d (\lambda_k^*)^{q^N} \lambda_k
		= -\Psi_N(1/\nfk)
		= 1 - \prod_{a\in A_{+,N}} \frac{\nfk a+1}{\nfk a}
		$$
		are explicit rational numbers in $k$.
		Thus, taking reduction, one computes that the $(i,j)$-entry of the product in \eqref{eq-product-in-residue-and-trace} is
		$$
		\sum_{k=1}^d \lambda_k^{q^{i-1}} (\lambda_k^*)^{q^{j-1}}
		=
		\begin{cases}
			-\Psi_0(1/\nfk)^{q^{i-1}} \equiv 1 \pmod{\Pfk}, & i = j, \\
			-\Psi_{j-i}(1/\nfk)^{q^{i-1}}  \equiv 0 \pmod{\Pfk}, & i < j,  \\
			 -\Psi_{d+j-i}(1/\nfk)^{q^{i-1}} \equiv 0 \pmod{\Pfk}, & i > j.
		\end{cases}
		$$
		In other words,
		\begin{equation}      \label{eq-result-in-residue-and-trace}
		\left( \lambda_j^{q^{i-1}} \right)
		\left( (\lambda_i^*)^{q^{j-1}} \right)
		\equiv I_d  \pmod{\Pfk}.
		\end{equation}
		Exchanging the product, and the result is reflected in the $(i,j)$-entry.
	\end{proof}
	
	\begin{rem}
		Note that \eqref{eq-result-in-residue-and-trace} coincides with the system of linear equations \eqref{eq-linear-equations-MZ-I} considered in the proof of Proposition \ref{prop-lambdai*-and-trace-pairing}.
		This already suggests that $e^*(a_i/\nfk)^q$ can be obtained from Ore's formula.
		Note that the trick of exchanging the product of two matrices was also used during that proof.
	\end{rem}
		
	\subsubsection{Review of Poonen pairing}
	
	We now briefly review the notion of \textit{Poonen pairing}.
	To ease the exposition and focus on our main goal, we restrict ourselves to one particular example, the Carlitz module, and refer to \cite{poonen1996fractional} or \cite[Subsection 4.14]{goss1996basic} for those interested readers.
	We will also omit the justifications of all basic properties we mention for the same reason.
	Let $\nfk \in A_+$ be arbitrary.
	
	\begin{defn}   \label{defn-poonen-pairing}
		For $a\in \Lambda_\nfk^*$ and $b\in \Lambda_\nfk$, write
		$$
		a\tau^0 C_\nfk(\tau) = (\tau^0 - \tau)h_a(\tau).
		$$
		Then the \textit{Poonen pairing with respect to $\nfk$} is defined as
		$$
		\ang{\cdot,\cdot}_{\textnormal{Poon}(\nfk)}: \Lambda_\nfk^* \times \Lambda_\nfk \to \Fq, \quad \ang{a,b}_{\textnormal{Poon}(\nfk)} := h_a(b).
		$$
	\end{defn}
	
	This is a non-degenerate, bilinear, and Galois-equivariant pairing on the torsion points of the Carlitz and adjoint Carlitz modules.
	Moreover, it is also compatible with the $A$-module structures in the sense that the equation
	$$
	\ang{C_f^*(a),b}_{\textnormal{Poon}(\nfk)} = \ang{a,C_f(b)}_{\textnormal{Poon}(\nfk)}
	$$
	holds for all $f\in A$.
	Hence, it is viewed as an analog of Weil pairing on elliptic curves.
	(This will not be needed in the sequel but is meant to be pointed out to illustrate its analogy with the classical situation.)
	
	Another compatibility of the Poonen pairings with respect to different $\nfk$ will also be useful for us.
	
	\begin{prop}    \label{prop-compatibility-of-poonen-pairing}
		Let $\nfk,\mfk \in A_+$ be given.
		Suppose $a\in \Lambda_{\nfk}^* \sbe \Lambda_{\nfk\mfk}^*$ and $b \in \Lambda_{\nfk\mfk}$. Then the equation
		$$
		\ang{a,b}_{\textnormal{Poon}(\nfk\mfk)} = \ang{a,C_\mfk(b)}_{\textnormal{Poon}(\nfk)}
		$$
		is true.
	\end{prop}
	
	\subsubsection{Residue pairing and Poonen pairing}
	
	There is an analogous result to Proposition \ref{prop-lambdai*-and-trace-pairing} for the Poonen pairing.
	
	\begin{prop}     \label{prop-lambdai*-and-poonen-pairing}
		For $\nfk \in A_+$, let $\{\lambda_j\}_{j=1}^{\deg\nfk}$ be any $\Fq$-basis of $\Lambda_\nfk$.
		Take $\{\lambda_i^*\}_{i=1}^{\deg\nfk} \sbe \Lambda_\nfk^*$ so that $\ang{\lambda_i^*, \lambda_j}_{\textnormal{Poon}(\nfk)} = - \delta_{ij}$.
		Then one has
		$$
		\lambda_i^* = (-1)^{\deg\nfk+i} \left(\frac{\Delta_i}{\Delta}\right)^q.
		$$
		In other words, the $\lambda_i^*$ given by Ore's formula are the dual bases of $\lambda_j$ with respect to Poonen pairing (with a minus sign modification).
	\end{prop}
	
	\begin{proof}
		Fix $1\leq i \leq \deg\nfk$. As $a_i := \lambda_i^*$ satisfies $C_\nfk^*(z)$, we have $a_i^{(1-q)/q} \tau^0 - \tau$ left divides $C_\nfk(z)$ (see \cite[Corollary 1.7.7]{goss1996basic}).
		So we may write
		$$
		C_\nfk(\tau) = \left(a_i^{(1-q)/q} \tau^0 - \tau\right) Q_i(\tau)
		$$
		for some $Q_i(\tau) \in \ovl{k}\{\tau\}$.
		Note that
		\begin{equation}   \label{eq-apply-poonen-pairing}
			a_i\tau^0 C_\nfk(\tau)
			= a_i\tau^0 \left(a_i^{(1-q)/q} \tau^0 - \tau\right) Q_i(\tau)
			= (\tau^0-\tau) a_i^{1/q} Q_i(\tau).
		\end{equation}
		So by Definition \ref{defn-poonen-pairing}, we have
		$$
		-\delta_{ij}
		= \ang{a_i, \lambda_j}_{\textnormal{Poon}(\nfk)}
		= a_i^{1/q} Q_i(\lambda_j)
		\quad
		\text{for all}
		\quad
		1 \leq j \leq \deg\nfk.
		$$
		In particular, $Q_i(\lambda_j) = 0$ for all $j\neq i$.
		So the subspace $W_i \sbe \Lambda_\nfk$ of roots of $Q_i(z) = 0$ is spanned by $\{\lambda_1,\ldots,\lambda_{i-1},\lambda_{i+1},\ldots,\lambda_{\deg\nfk}\}$.
		The leading coefficient of $Q_i(z)$ is seen to be $-1$ from \eqref{eq-apply-poonen-pairing}. 
		Thus, we may write
		$$
		Q_i(z) = - \prod_{\lambda \in W_i} (z - \lambda) =: -P_i(z).
		$$
		Also, note that
		$$
		-1 
		= \ang{a_i, \lambda_i}_{\textnormal{Poon}(\nfk)}
		= a_i^{1/q} Q_i(\lambda_i)
		= -a_i^{1/q} P_i(\lambda_i)
		\implies
		a_i = P_i(\lambda_i)^{-q}.
		$$
		A property of the Moore determinant (see \cite[Theorem 1.3.5.2]{goss1996basic}) now implies that
		$$
		\lambda_i^* 
		= a_i
		= P_i(\lambda_i)^{-q}
		= \left(\frac{\Delta(\lambda_1,\ldots,\lambda_{i-1},\lambda_{i+1},\ldots,\lambda_{\deg\nfk})}{\Delta(\lambda_1,\ldots,\lambda_{i-1},\lambda_{i+1},\ldots,\lambda_{\deg\nfk},\lambda_i)}\right)^q
		= (-1)^{\deg\nfk+i} \left(\frac{\Delta_i}{\Delta}\right)^q.
		$$
		This completes the proof.
	\end{proof}
	
	The relation between $\nfk$-dual families $\{a_i\}_{i=1}^{\deg\nfk}, \{b_j\}_{j=1}^{\deg\nfk}$ and the torsion points
	$$
	\{e^*(\alpha' a_i/\nfk)^q\}_{i=1}^{\deg\nfk},
	\quad
	\{e(\alpha b_j/\nfk)\}_{j=1}^{\deg\nfk}
	$$
	is also explained by the Poonen pairing. Precisely, we have
	
	\begin{prop}    \label{prop-residue-and-poonen}
		Let $\nfk \in A_+$ and $\{a_i\}_{i=1}^{\deg\nfk}, \{b_j\}_{j=1}^{\deg\nfk}$ be $\nfk$-dual families.
		Take $\alpha \in A$ prime to $\nfk$ and let $\alpha' \in A$ be such that $\alpha\alpha' \equiv 1 \pmod{\nfk}$.
		Then one has
		$$
		\ang{e^*(\alpha' a_i/\nfk)^q, e(\alpha b_j/\nfk)}_{\textnormal{Poon}(\nfk)} = - \delta_{ij}.
		$$
		In other words, the functions $e^*$ and $e$ map dual bases of the residue pairing to dual bases of the Poonen pairing (again, with a minus sign modification).
	\end{prop}
	
	Note that Propositions \ref{prop-lambdai*-and-poonen-pairing} and \ref{prop-residue-and-poonen} imply the main goal \eqref{eq-the-equation-of-the-main-goal} for the general case.
	
	\begin{proof}
		As $\{\alpha' a_i\}_{i=1}^{\deg\nfk}, \{\alpha b_j\}_{j=1}^{\deg\nfk}$ is also a pair of $\nfk$-dual families, we may assume without loss of generality that $\alpha = 1$.
		First, we consider the special case $\nfk = v - 1$ where $v \in A_{+,d}$ is irreducible.
		Let $\Pfk$ be a prime in $K_\nfk$ above $v$.
		In the definition of Poonen pairing, one considers everything modulo $\Pfk$ in the residue field $\FF_\Pfk$ and defines the “reduction” of Poonen pairing.
		More precisely, put (recall Propositions \ref{prop-reduction-of-lambda} and \ref{prop-reduction-of-lambda*})
		$$
		\ovl{\Lambda}_\nfk := \{ \ovl{\lambda} \mid \lambda \in \Lambda_\nfk \}
		\quad
		\text{and}
		\quad
		\ovl{\Lambda}{}^*_\nfk := \{ \ovl{\lambda}{}^* \mid \lambda^* \in \Lambda_\nfk^* \}.
		$$
		For $\ovl{a} \in \ovl{\Lambda}{}^*_\nfk$ and $\ovl{b} \in \ovl{\Lambda}_\nfk$, write
		$$
		\ovl{a} \tau^0 \ovl{C}_\nfk(\tau) = (\tau^0 - \tau) \ovl{h}_{\ovl{a}}(\tau)
		$$
		where $\ovl{C}_\nfk(\tau)$ is the twisted polynomial obtained from reducing the coefficients of $C_\nfk(\tau)$ modulo $v$.
		Then the reduced Poonen pairing (with respect to $\nfk = v-1$) is defined as
		$$
		\ang{\cdot,\cdot}_{\ovl{\textnormal{Poon}}(\nfk)}: \ovl{\Lambda}{}^*_\nfk \times \ovl{\Lambda}_\nfk \to \Fq,
		\quad 
		\ang{\ovl{a},\ovl{b}}_{\ovl{\textnormal{Poon}}(\nfk)}
		:= \text{the unique lift of }
		\ovl{h}_{\ovl{a}} (\ovl{b}) \sbe \FF_{\Pfk}
		\text{ to }
		\Fq.
		$$
		
		Note that
		$$
		\ovl{h}_{\ovl{a}} (\ovl{b})
		= \ovl{h_a(b)}
		= \ovl{\ang{a,b}_{\textnormal{Poon}(\nfk)}}.
		$$
		So in fact,
		\begin{equation}    \label{eq-reduced-poonen-and-poonen}
			\ang{\ovl{a},\ovl{b}}_{\ovl{\textnormal{Poon}}(\nfk)}
			= \ang{a,b}_{\textnormal{Poon}(\nfk)}.
		\end{equation}
		On the other hand, as $\nfk = v-1$, we know $\ovl{C}_\nfk (\tau) = \tau^d - \tau^0$.
		This implies that
		$$
		\ovl{h}_{\ovl{a}} (\tau) = -(\tau^0 + \tau^1 + \cdots + \tau^{d-1}) \ovl{a}\tau^0
		\implies
		\ovl{h}_{\ovl{a}} (\ovl{b})
		= -\sum_{k=0}^{d-1} (\ovl{ab})^{q^k}.
		$$
		Thus, using the trace pairing considered in \S\ref{subsubsection-residue-pairing-and-trace-pairing}, we have
		\begin{equation}       \label{eq-reduced-poonen-and-trace}
			\ang{\ovl{a},\ovl{b}}_{\ovl{\textnormal{Poon}}(\nfk)}
			= -\angtr{\ovl{a},\ovl{b}}.
		\end{equation}
		Hence, by Proposition \ref{prop-residue-and-trace}, \eqref{eq-reduced-poonen-and-trace}, and \eqref{eq-reduced-poonen-and-poonen}, we have
		$$
		\delta_{ij}
		= \angtr{\ovl{e^*(a_i/\nfk)^q},\ovl{e(b_j/\nfk)}}
		= -\ang{e^*(a_i/\nfk)^q,e(b_j/\nfk)}_{\textnormal{Poon}(\nfk)}.
		$$
		This completes the case where $\nfk = v-1$.
		
		For general $\nfk \in A_+$, we apply Dirichlet density theorem (see \cite[Chapter 4]{rosen2002number}) to take an irreducible $v \in A_+$ such that $v \equiv 1 \pmod{\nfk}$.
		Write $\nfk\mfk = v - 1 =: \nfk'$ for some $\mfk \in A_+$.
		Note that for the pair $\{a_i\mfk\}_{i=1}^{\deg\nfk}, \{b_j\}_{j=1}^{\deg\nfk}$, we have
		$$
		\Res(a_i\mfk \cdot b_j / \nfk') = \Res(a_ib_j/\nfk) = \delta_{ij}.
		$$
		So we may extend it to a pair of $\nfk'$-dual families $\{a_i'\}_{i=1}^{\deg\nfk'}, \{b_j'\}_{j=1}^{\deg\nfk'}$ so that
		$$
		a_i' = a_i\mfk
		\quad
		\text{and}
		\quad
		b_j' = b_j
		\quad
		\text{for all}
		\quad
		1\leq i,j \leq \deg\nfk.
		$$
		Then for any such $i,j$, we have by the previous case and Proposition \ref{prop-compatibility-of-poonen-pairing} that
		\begin{align*}
			-\delta_{ij}
			&= \ang{e^*(a_i'/\nfk')^q,e(b_j'/\nfk')}_{\textnormal{Poon}(\nfk')}
			= \ang{e^*(a_i/\nfk)^q,e(b_j'/\nfk')}_{\textnormal{Poon}(\nfk\mfk)}    \\
			&= \ang{e^*(a_i/\nfk)^q,C_\mfk(e(b_j'/\nfk'))}_{\textnormal{Poon}(\nfk)}
			= \ang{e^*(a_i/\nfk)^q, e(b_j/\nfk)}_{\textnormal{Poon}(\nfk)}.
		\end{align*}
		This proves the general case.
	\end{proof}
	
	In the first half of the proof, we may consider $\mfk = v^\l - 1$ for any $\l \in \NN$ and still obtain the identities \eqref{eq-reduced-poonen-and-poonen} and \eqref{eq-reduced-poonen-and-trace}, where the trace is taken from $\FF_{\Pfk} \simeq \Fqdl$ to $\Fq$.
	Hence, to summarize, we have proved the following theorem.
	
	\begin{thm}    \label{thm-pairing-summary}
		Let $\mfk := v^\l-1$ and $\{a_i\}_{i=1}^{d\l}, \{b_j\}_{j=1}^{d\l}$ be any $\mfk$-dual families.
		Let $\lambda_i^* := e^*(\alpha' a_i/\mfk)^q$ and $\lambda_j := e(\alpha b_j/\mfk)$ where $\alpha'\alpha \equiv 1 \pmod{\mfk}$.
		Then we have
		$$
		\angres{a_i,b_j}
		= - \ang{\lambda_i^*,\lambda_j}_{\textnormal{Poon}(\mfk)}
		= \angtr{\ovl{\lambda}{}^*_i,\ovl{\lambda}_j}
		= \delta_{ij}
		$$
		(these three pairings correspond respectively to the isomorphisms in Propositions \ref{prop-reduction-of-lambda} and \ref{prop-reduction-of-lambda*})
		and
		$$
		\lambda_i^* = (-1)^{\deg\nfk+i} \left(\frac{\Delta_i}{\Delta}\right)^q
		$$
		where $\Delta_i := \Delta(\lambda_1,\ldots,\lambda_{i-1},\lambda_{i+1},\ldots,\lambda_{d\l})$ and $\Delta := \Delta(\lambda_1,\ldots,\lambda_{d\l})$ are the Moore determinants.
	\end{thm}
	
	\subsection{The behavior of geometric Gauss sums at infinity}    \label{subsection-the-behavior-of-geometric-gauss-sums-at-infinity}
	
	We now resume the study of geometric Gauss sums at infinity.
	Recall in Proposition \ref{prop-absolute-values}, we proved that their (normalized) valuations at any infinite place of $K_{\nfk,d\l}$ is $-1/(q-1)$.
	By our identification from $\ovl{k}$ into $\ovl{k}_\infty$, we may regard $\bggs_x$ as elements in $\ovl{k}_\infty$ and consider their signs at infinity.
	We choose a sign function so that $\sgn(\td{\pi})^{q-1} = -1$ (cf. \eqref{eq-period}).
	By choosing a suitable embedding, we assume without loss of generality that $\lambda = e(1/\nfk)$.
	Set $\mfk := v^\l-1$ as in \S\ref{subsection-a-scalar-product-expression}.
	
	\begin{prop}   \label{prop-sign}
		Let $\epsilon := \sgn(\td{\pi})$.
		For $x \in \nfk^{-1}A$ with $0 < |x|_\infty < 1$ and $0 \leq s \leq d\l-1$, we have
		$$
		\sgn((\bggs_x)^{\tau_q^s}) = -\epsilon \psi\left( \ovl{e^*(x)^q} \right)^{q^s}.
		$$
		In particular,
		$$
		\sgn(\ggs(x)) = (-\epsilon)^{d\l} \Nr_{\FF_{\Pfk}/\Fq} \left( \ovl{e^*(x)^q} \right).
		$$
	\end{prop}
	
	\begin{proof}
		Let $\Sigma := (\bggs_x)^{\tau_q^s} - 1$.
		Note by Proposition \ref{prop-absolute-values} and the non-Archimedean property, we have $\ord_{\td{\infty}} (\Sigma) = \ord_{\td{\infty}}((\bggs_x)^{\tau_q^s}) = -1/(q-1)$.
		Thus, it suffices to determine the sign of $\Sigma$.
		
		Write $x = a_0/\mfk$ with $\deg a_0 < d\l$.
		Since every element in $\FF_{\Pfk}^\times$ is represented by $e(a/\mfk) \in \Lambda_{\mfk}$ for some unique $0 \neq a \in A/\mfk$, we have
		$$
		\Sigma
		= \sum_{z \in \FF_{\Pfk}^\times} \omega\left(C_{\mfk x} (z^{-1})\right) \psi(z)^{q^s}
		= \sum_{0 \neq a \in A/\mfk} e(a_0a/\mfk) \psi\left( \ovl{e(a/\mfk)}^{-1} \right)^{q^s}.
		$$
		For each such $a$, we let $a'$ be the unique element in $A$ such that $a' \equiv a_0a \pmod{\mfk}$ and $\deg a' < d\l$.
		Then from the infinite product expression
		$$
		e(z) = \td{\pi}z \prod_{0 \neq a \in A} \left( 1+\frac{z}{a} \right),
		$$
		we see that
		\begin{align*}
			&e(a_0a/\mfk) = e(a'/\mfk) \text{ has minimal valuation}   \\
			\iff{} &\ord_{\infty}(a'/\mfk) = d\l - \deg a' \text{ is minimal}   \\
			\iff{} &\deg a' = d\l-1     \\
			\iff{} &\Res(a'/\mfk) = \Res(a_0a/\mfk)\neq 0.
		\end{align*}
		And in this case, $\sgn(e(a_0a/\mfk)) = \epsilon \Res(a_0a/\mfk)$.
		Thus, we have
		$$
		\sgn(\Sigma)
		= \epsilon \sum_{0 \neq a \in A/\mfk} \Res(a_0a/\mfk) \psi\left( \ovl{e(a/\mfk)}^{-1} \right)^{q^s}.
		$$
		Put $\lambda_a := e(a/\mfk)$ and $\lambda_{a_0}^* := e^*(a_0/\mfk)^q$.
		Then by the compatibility of the residue pairing and the trace pairing (Theorem \ref{thm-pairing-summary}), this implies that
		\begin{align*}
			\sgn(\Sigma)
			&= \epsilon \sum_{0 \neq a \in A/\mfk} \Tr_{\FF_\Pfk/\Fq} (\ovl{\lambda}{}^*_{a_0} \ovl{\lambda}_a) \psi \left( \ovl{\lambda}_a^{-1} \right)^{q^s}   \\
			&= \epsilon \sum_{0 \neq b \in A/\mfk} \Tr_{\FF_\Pfk/\Fq}(\ovl{\lambda}_b) \psi\left( \ovl{\lambda}{}^*_{a_0} \ovl{\lambda}_b^{-1} \right)^{q^s}
			= \epsilon \psi\left( \ovl{\lambda}{}^*_{a_0} \right)^{q^s} \sum_{z \in \FF_{\Pfk}^\times} \Tr_{\FF_\Pfk/\Fq}(z) \psi(z^{-1})^{q^s}.
		\end{align*}
		By Lemma \ref{lem-descending-ggs}, the last sum is seen to be $-1$.
		This completes the proof.
	\end{proof}
	
	\printbibliography
	
\end{document}